\documentclass[11pt,twoside,leqno]{article}

\usepackage{amssymb}
\usepackage{amsmath}
\usepackage{amsthm}
\usepackage{graphicx}
\usepackage{color}

\allowdisplaybreaks

\def\rr{{\mathbb R}}
\def\rn{{{\rr}^n}}

\def\zz{{\mathbb Z}}
\def\nn{{\mathbb N}}

\def\cc{{\mathbb C}}
\def\cs{{\mathcal S}}

\def\cm{{\mathcal M}}
\def\cn{{\mathcal N}}

\def\fz{\infty}
\def\az{\alpha}

\def\dist{{\mathop\mathrm{\,dist\,}}}

\def\lz{\lambda}

\def\gz{{\gamma}}
\def\oz{{\omega}}

\def\vz{\varphi}

\def\sz{\sigma}

\def\wz{\widetilde}

\def\hs{\hspace{0.3cm}}

\def\ls{\lesssim}

\def\gfz{\genfrac{}{}{0pt}{}}

\def\rn{{{\mathbb R}^n}}
\def\rr{{\mathbb R}}
\def\cc{{\mathbb C}}
\def\zz{{\mathbb Z}}
\def\nn{{\mathbb N}}
\def\cm{{\mathcal M}}

\def\hs{\hspace{0.3cm}}

\def\fz{\infty}
\def\az{\alpha}
\def\supp{{\mathop\mathrm{\,supp\,}}}
\def\dist{{\mathop\mathrm{\,dist\,}}}
\def\diam{{\mathop\mathrm{\,diam\,}}}
\def\lz{\lambda}

\def\gz{{\gamma}}
\def\vz{\varphi}

\def\sz{\sigma}
\def\wz{\widetilde}

\def\ls{\lesssim}

\def\oz{\overline}

\def\diam{{\mathop\mathrm{\,diam\,}}}

\def\dfrac{\displaystyle\frac}
\def\dsup{\displaystyle\sup}

\def\r{\right}
\def\lf{\left}

\newcommand{\dbt}{\dot{B}_{p,q}^{s,\tau}(\rn)}

\newcommand{\bt}{{B}_{p,q}^{s,\tau}(\rn)}

\newtheorem{theorem}{Theorem}[section]
\newtheorem{lemma}[theorem]{Lemma}
\newtheorem{corollary}[theorem]{Corollary}
\newtheorem{proposition}[theorem]{Proposition}

\theoremstyle{definition}
\newtheorem{remark}[theorem]{Remark}
\newtheorem{definition}[theorem]{Definition}
\numberwithin{equation}{section}

\def\hs{\hspace{0.3cm}}

\begin{document}

\title{\bf\Large The Radial Lemma of Strauss in the Context of Morrey Spaces
\footnotetext{\hspace{-0.35cm} 2010 {\it
Mathematics Subject Classification}. Primary  46E35, Secondary 42C15.
\endgraf
{\it Key words and phrases}.  Smoothness of radial function, decay of
radial function, Sobolev-Morrey space,  Besov-type space, Besov-Morrey space.  \endgraf
This work is partially supported by the 2010 Joint Research Project Between
China Scholarship Council and German Academic Exchange
Service (PPP) (Grant No. LiuJinOu [2010]6066).
Wen Yuan is supported by the Alexander von
Humboldt Foundation and the National
Natural Science Foundation (Grant No. 11101038) of China,
and Dachun Yang is supported by the National
Natural Science Foundation (Grant No. 11171027) of China
and Program for Changjiang Scholars and Innovative
Research Team in University of China.
}}
\date{}
\author{Wen Yuan, Winfried Sickel and Dachun Yang\footnote{Corresponding author}}
\maketitle

\vspace{-0.6cm}

\begin{center}
\begin{minipage}{13cm}{\small
{\noindent{\bf Abstract}\quad
In this paper we consider smoothness and decay properties of radial functions
belonging to smoothness spaces related to Morrey spaces (Sobolev-Morrey spaces,
Besov-type spaces and Besov-Morrey spaces).
Within this framework we prove generalizations of the Radial Lemma of Strauss.
}}
\end{minipage}
\end{center}

\arraycolsep=1pt

\vspace{0.1cm}


\section{Introduction}


\hskip\parindent At the end of the seventies Strauss \cite{St}
was the first who  observed that
there is an interplay between the regularity and decay properties of radial functions.
We recall his lemma as follows.

\begin{proof}[\rm\bf Radial Lemma] \it Let $n\ge 2$.
Every radial function $f \in H^1 (\rr^n)(=W^1_2(\rn))$ is almost everywhere equal to a function
$\widetilde{f}$, continuous for $x\neq 0$, such that
\begin{equation}\label{strauss}
|\widetilde{f} (x)| \le C \, |x|^{\frac{1-n}{2}}\, \| \,f \, \|_{H^1(\rr^n)}\, ,
\end{equation}
where the positive constant $C$ depends only on $n$.
\end{proof}

Strauss stated (\ref{strauss}) with the extra condition $|x|\ge 1$, but this restriction is not needed.
The {\em Radial Lemma} contains three different assertions:
\vspace{-0.25cm}
\begin{itemize}
 \item[(a)]
 the existence of a representative of $f$, which is
continuous outside the origin;
\item[(b)]\vspace{-0.28cm}
the decay of $f$ near infinity;
\item[(c)]\vspace{-0.28cm}
the limited unboundedness near the origin.
\end{itemize}
\vspace{-0.25cm}
These three  properties do not extend to all functions in $H^1(\rr^n)$, of course.
In particular, $H^1 (\rr^n) \not \subset L_\infty (\rr^n)$, $n\ge 2$.
Taking an unbounded function $f\in H^1 (\rr^n)$, say, unbounded near the origin,
by means of a simple shift $f(x-x^0)$ we obtain a function which is now
unbounded around $x^0$ and belongs to $H^1 (\rr^n)$ as well.
This makes also clear that an inequality like (\ref{strauss}) can not  be true in the general context.

Later on, Lions \cite{Li} showed how the {\em Radial Lemma} extends to first order Sobolev spaces.

\begin{proposition}\label{decay1}
 Let $n\ge 2$ and $1 \le p < \infty$.
Every radial function $f \in W^1_p (\rr^n)$ is almost everywhere equal to a function
$\widetilde{f}$, continuous for $x\neq 0$, such that
\begin{equation}\label{lions}
|\widetilde{f} (x)| \le C\, |x|^{\frac{1-n}{p}}\, \| \,f \, \|_{W^1_p(\rr^n)}\, ,
\end{equation}
where the positive constant $C$ depends only on $n$ and $p$.
\end{proposition}

In recent years, after the seminal paper of Kozono and Yamazaki \cite{KY}, an increasing number of papers
on existence and regularity of solutions of certain partial differential equations is using
Morrey spaces and smoothness spaces built on Morrey spaces (as Sobolev-Morrey spaces, Besov-Morrey spaces or
Besov-type spaces).
For convenience of the readers, we recall their definitions (also to fix the notation which is different in different papers).

\begin{definition}\label{d1.2}
{\rm (i)} Let  $0<p\le u\le\infty$.
The \emph{Morrey space } $\mathcal{M}^u_{p}(\rr^n)$
is defined to be the set of all $p$-locally
Lebesgue-integrable functions $f$ on $\rr^n$ such that
\[
\|f\|_{\mathcal{M}^u_{p}(\rr^n)} :=  \sup_{B}
|B|^{1/u-1/p}\lf[\int_B |f(x)|^p\,dx\r]^{1/p}<\infty\, ,
\]
where the supremum is taken over all balls $B$ in $\rr^n$.

{\rm (ii)}
Let $m \in \nn$ and $1 \le p \le u \le \infty$. The \emph{Sobolev-Morrey space}
$W^m \cm^u_p (\rr^n)$ is the collection of all functions $f \in \cm^u_p (\rr^n)$ such that all distributional derivatives
$D^\alpha f$ of order $|\alpha|\le m$ belong to $\cm_p^u (\rr^n)$. We equip this space with the norm
\[
\|f\|_{W^m \cm^u_p (\rr^n)} := \sum_{|\alpha|\le m} \|D^\alpha f\|_{\cm_p^u (\rr^n)}\, .
\]
\end{definition}

\begin{remark}
We call the parameter $u$ in Definition \ref{d1.2}
the \emph{Morrey parameter} (also in connection with the
Besov-type and Besov-Morrey spaces defined below).
\end{remark}

Obviously we have $\mathcal{M}^p_{p} (\rr^n) = L^p (\rr^n)$
and $\mathcal{M}^\infty_{p}(\rr^n)= L^\infty(\rr^n).$
As a consequence of H\"older's inequality we conclude monotonicity with respect to $p$,
i.\,e.,
\begin{equation*}
\cm^u_w (\rr^n)\hookrightarrow \cm^u_p (\rr^n) \qquad
\mbox{if}\qquad 0 < p \le  w \le u \le \infty\, .
\end{equation*}
In particular $L^p(\rr^n) = \cm^p_p (\rr^n)\hookrightarrow \cm^p_v (\rr^n)$ if
$ v< p$. This implies that
\begin{equation*}
W^m_p (\rn) \hookrightarrow W^m \cm^p_v (\rr^n)\quad\mathrm{if}\quad 1\le  v \le p \le \infty \, .
\end{equation*}
For a better understanding and later use we mention also
the embedding into  the \emph{class
of all complex-valued, uniformly continuous and bounded functions
on $\rn$}, here denoted by $C_{ub} (\rn)$.
We have
\begin{equation}\label{morrey2}
W^{m} \cm^{u}_{p}(\rn) \hookrightarrow C_{ub} (\rn)
\qquad \mbox{if} \qquad \frac nu <   m \, ;
\end{equation}
see \cite[Proposition 17]{s011}.
Our first  result  consists in an extension of
Proposition \ref{decay1} to Sobolev spaces built on Morrey spaces.

\begin{theorem}\label{t3.2b}
Let $n\ge 2$ and $1\le p \le  u \le  \fz$.
Every radial function $f \in W^1 \cm_p^u (\rr^n)$ is almost everywhere equal to a function
$\widetilde{f}$, continuous for $x\neq 0$, such that
\begin{equation}
\label{eq-erg1b}
|\widetilde{f}(x)|\le C \, |x|^{\frac{1}p - \frac{n}{u}} \,  \|f\|_{W^{1} \cm^{u}_{p}(\rn)}\,
\end{equation}
holds  for all $|x|\ge 1$ with some positive
constant $C$ depending only on $n$, $p$ and $u$.
\end{theorem}

\begin{proof}
The proof is simple and (\ref{eq-erg1b}) can be easily reduced to (\ref{lions}).
To see this, let $f \in W^1 \cm_p^u (\rr^n)$ be a radial function.
We choose a smooth radial cut-off function
$$\varrho \in RC_0^\infty (\rr^n):=\{f\in C_0^\fz(\rn):\ f\ \mbox{is radial}\}$$
such that $0\le\varrho\le1$, $\supp \varrho \subset \{x\in \rr^n: \:
1/2 <|x| < 2\}$,  and
$\sum_{j\in\zz} \varrho(2^{-j}x)=1$ for all $x\in\rn\setminus\{0\}$.
Then, for all $2^k\le |x|<2^{k+1} $ with $k \in \zz$,
we see that
$$f(x)=\sum_{j\in\zz}\varrho(2^{-j}x)f(x)=\sum_{j=k-1}^{k+1}\rho(2^{-j}x)f(x).$$
Notice that for all $k\in\nn_0:=\{0,1,\ldots\}$,
the function $\sum_{j=k-1}^{k+1}
\varrho (2^{-j}\, \cdot \, ) \, f (\, \cdot \,)$ is a radial function
in $W^1_p (\rr^n)$ and hence by Proposition \ref{decay1},
there exists a function $g_k$, which is continuous on $\rn\setminus\{0\}$ and
coincides with
$\sum_{j=k-1}^{k+1}
\varrho (2^{-j}\, \cdot \, ) \, f (\, \cdot \,)$
almost everywhere, such that
for all $x\in\rn$,
\begin{eqnarray*}
|g_k(x)|   & \le &  C \, |x|^{\frac{1-n}{p}}\,\sum_{j=k-1}^{k+1} \|
\,\varrho (2^{-j}\, \cdot \, ) \, f (\, \cdot \,) \,\|_{W^1_p(\rr^n)}\\
& \le  &  C \, |x|^{\frac{1-n}{p}}\, \sum_{|\alpha|\le 1}\,
\lf[\int_{2^{k-2}<|y|< \, 2^{k+2}}
\,|\, D^ \alpha (\varrho (2^{-j}\, \cdot) \, f)(y) \, |^p \, dy\r]^{1/p}\\
& \le  &  C \, \lf[\max_{|\alpha|\le 1} \| D^\alpha \varrho\,\|_{L^\infty (\rr^n)}\r]\,
|x|^{\frac{1-n}{p}}\, \sum_{|\alpha|\le 1}\, \lf[\int_{2^{k-2}< |y|<\, 2^{k+2}}
\,|\, D^ \alpha \, f (y) \, |^p \, dy\r]^{1/p} \, ,
\end{eqnarray*}
where $C$ is a positive constant depending only on $p$ and $n$.
Now we switch to the Morrey norm by inserting the volume of the ball with radius
$2^{k+2}$ to the power $1/u - 1/p$. By $\omega_n$ we denote the
\emph{volume of the unit ball}.
Then it follows that, for all $2^k \le|x| <2^{k+1} $ with $k\in \nn_0$,
\begin{eqnarray*}
|g_k(x)| & \le &  C \, \lf[\max_{|\alpha|\le 1} \| D^\alpha \varrho\, \|_{L^\infty (\rr^n)}\r]\,
|x|^{\frac{1-n}{p}}\, \lf(4^n\, 2^{kn}\, \omega_n\r)^{\frac 1p - \frac 1u} \,
\|\, f  \, \|_{W^ 1 \cm_p^u (\rr^n)}\\
& \le &
C \, |x|^{\frac{1}{p} - \frac{n}{u}}\,
\|\, f  \, \|_{W^ 1 \cm_p^u (\rr^n)} \, ,
\end{eqnarray*}
where $C$ is a positive constant depending on $p,u$ and $n$ only.

Now let $\wz{f}:=\sum_{k\in \zz} g_k/3$. Since
$\supp g_k\cap \supp g_j=\emptyset$
if $|j-k|\ge 5$, the summation over $k$ is finite, and hence
$\wz{f}$ is continuous on $\rn\setminus\{0\}$ and satisfies that
for all $|x|\ge1$,
$$|\wz{f}(x)|\le C \, |x|^{\frac{1}{p} - \frac{n}{u}}\,
\|\, f  \, \|_{W^ 1 \cm_p^u (\rr^n)} \,,$$
where $C$ is a positive constant depending on $p,u$ and $n$ only.
Moreover, since $g_k$ coincides with
$\sum_{j=k-1}^{k+1}
\varrho (2^{-j}\, \cdot \, ) \, f (\, \cdot \,)$ almost everywhere, we see that
$$\wz{f}(x)=\sum_{k\in\zz}
\sum_{j=k-1}^{k+1}\varrho(2^{-j}x)f(x)/3=\sum_{j\in\zz}\varrho(2^{-j}x)
\sum_{k=j-1}^{j+1}f(x)/3=f(x)$$
almost every $x\in\rn$.
This finishes the proof of Theorem \ref{t3.2b}.
\end{proof}

\begin{remark}
(i) Obviously the inequality (\ref{eq-erg1b})  covers (\ref{lions}) (for $u=p$).
Furthermore, we have decay near infinity if $1/p < {n}/{u}$.
In case $1/p > n/u$ the inequality (\ref{eq-erg1b}) can be immediately improved by (\ref{morrey2})
resulting in global boundedness of all elements of $W^{1} \cm^{u}_{p}(\rn)$.

(ii) The elementary proof given above shows that we do not need all available information
on the elements of $W^{1} \cm^{u}_{p}(\rn)$.
In fact, we only need that $f $ belongs to $W^{1}_{p}(\rn)$ locally and that
\[
\sup_{B}
|B|^{1/u-1/p}\lf\{\int_B | D^\alpha f(x)|^p\,dx\r\}^{1/p}<\infty\, ,\qquad |\alpha|\le 1\, ,
\]
where the supremum is taken over all balls $B$, centered in the origin and having radius larger than $1$.
Such kind of spaces are usually
called central Morrey spaces or local Morrey  spaces. They attracted
some attention in recent years; see, for example,
\cite{alg,bn10}. By this point of view it
would make sense to study associated smoothness spaces.
However, for the moment we concentrate on smoothness spaces related to
the original Morrey norm.
\end{remark}

The main application of the Strauss lemma and its generalizations consists in the proof of
the compactness of the embedding  of the  radial subspace $RW^m_p (\rr^n)$
of the Sobolev space $W^m_p (\rr^n)$ into Lebesgue spaces $L^q (\rr^n)$
(see, for example, \cite{ss00,ssv,HS1,HS2,ss13}).
Those applications are possible also for this more general situation.
For this we refer to our forthcoming paper \cite{ysy2}.

In many papers, we refer, e.\,g., to the survey \cite{s011a} or the papers \cite{ss00},
\cite{s02}, \cite{CO}, \cite{ssv} and \cite{ss12},
it has been shown that one gets a better insight into the behavior of radial functions if one  replaces the  Sobolev space
$W^1_p (\rr^n)$  by spaces of fractional order of smoothness, for instance, Bessel potential or Besov  spaces.
In such a framework $H^1(\rn)$  can be replaced either by
$H^s(\rn)$ with $s>1/2$ or by $B^{1/2}_{2,1}(\rn)$ to guarantee the same conclusions as in the Strauss Lemma above.
But the Bessel potential space $H^s(\rn)$ with $1/2< s< 1$ and the Besov space
$B^{1/2}_{2,1}(\rn)$ are much larger than $H^1(\rn)$. Indeed, we have
\[
H^1(\rn)\hookrightarrow H^s(\rn)\hookrightarrow   B^{1/2}_{2,1}(\rn)\, , \qquad 1/2< s< 1\, ,
\]
and all the embeddings are strict.
The $p$-version looks like follows
\[
W^1_p(\rn)\hookrightarrow    B^{1/p}_{p,1}(\rn)\, , \qquad 1< p\le \infty \, ,
\]
and again all the embeddings are strict.
This motivates the use of the more complicated spaces of fractional order of smoothness,
which is also done here. Also in this much more general framework we know
\[
W^1 \cm^u_p(\rn) \hookrightarrow    B^{1/p, \frac 1p - \frac 1u}_{p,1}(\rn)\, , \qquad
1< p \le u \le \infty \, .
\]
The main aim of this paper consists in proving the extension of Theorem \ref{t3.2b}
to  the Besov-type spaces  $B^{1/p, \frac 1p - \frac 1u}_{p,1}(\rn)$
instead of the Sobolev-Morrey spaces.

The paper is organized as follows.
In Section \ref{fs} we recall the definition and some properties of
Besov-type spaces. The main part of this section is taken by
a flexibilization of the known characterizations of the spaces $\bt$ by atoms.
Section \ref{smooth} is devoted to  the investigation of the
regularity of radial functions belonging to $\bt$ outside the origin.
Finally, in Section \ref{decay}, we study the decay of radial functions.
As a service for the reader we reformulate the outcome for Sobolev-Morrey spaces in a separate subsection
at the end of Section \ref{decay}.

\subsection*{Notation}

\hskip\parindent As usual, $\nn$ denotes the \emph{natural numbers}, $\nn_0$
the \emph{natural numbers including $0$},
$\zz$ the \emph{integers} and
$\rr$ the \emph{real numbers}.
$\cc$ denotes the \emph{complex numbers} and $\rn$ the
\emph{$n$-dimensional  Euclidean space}.
All functions are assumed to be complex-valued, i.\,e.,
we consider functions
$f:~ \rn \to \cc$.
Let $\mathcal{S}(\rn)$ be the collection of all \emph{Schwartz functions} on $\rn$ endowed
with the usual topology and denote by $\mathcal{S}'(\rn)$ its \emph{topological dual}, namely,
the space of all bounded linear functionals on $\mathcal{S}(\rn)$
endowed with the weak $\ast$-topology.
The symbol $\widehat{\varphi}$ refers to  the \emph{Fourier transform} of $\vz\in\cs'(\rn)$.

All function spaces, which  we consider in this paper, are subspaces of $\cs'(\rn)$,
i.\,e. spaces of equivalence classes with respect to
almost everywhere equality.
However, if such an equivalence class  contains a continuous representative,
then usually we work with this representative and call also the equivalence class a continuous function.

If $E$ and
$F$ are two quasi-Banach spaces, then the \emph{symbol} $E \hookrightarrow F$
indicates that the embedding is continuous.
By $C_0^\infty(\rn)$ we denote the \emph{set
of all compactly supported and infinitely differentiable functions}.

Let $\mathcal{Q}$ be the collection of all \emph{dyadic cubes} in $\rn$, namely,
\begin{equation}\label{dyadic}
\mathcal{Q}:= \{Q_{j,k}:= 2^{-j}([0,1)^n+k):\ j\in\zz,\ k\in\zz^n\}.
\end{equation}
The \emph{symbol}  $\ell(Q)$ is used for
the side-length of $Q$ and $j_Q:=-\log_2\ell(Q)$.
For  $j\in\zz$,  $x\in\rn$,  $\vz \in \cs(\rn)$ and $Q\in\mathcal{Q}$,
let $\vz_j(x):= 2^{jn}\vz(2^jx)\, .$

As usual, the \emph{symbol}  $C $ denotes   a positive constant
which depends
only on the fixed parameters $n,s,\tau, p,q$ and probably on auxiliary functions,
unless otherwise stated; its value  may vary from line to line.
Sometimes we   use the symbol ``$ \ls $''
instead of ``$ \le $''. The \emph{meaning of $A \ls B$} is
given by: there exists a positive constant $C$ such that
 $A \le C \,B$.
The \emph{symbol $A \asymp B$} is used as an abbreviation of
$A \ls B \ls A$.


\section{Besov-type spaces}\label{fs}


\hskip\parindent Sobolev and Besov spaces
have been widely used in various areas of analysis such as harmonic analysis
and partial differential equations.
There is a rich literature
with respect to these spaces, we refer, e.\,g., to the monographs
\cite{BL,BIN78,BIN96,Pe76,t83,t92,t06}.

Here we would like to concentrate on Besov-type spaces which are related
to a relatively new family of function spaces, so-called $Q_\az$ spaces;
see, e.\,g., \cite{dx,ejpx,x,x06}.
These  spaces, in general, do not coincide with Besov-Morrey spaces
(Besov spaces built on Morrey spaces), but they are not so far from each other.
In the following subsection we recall definitions
and state some basic properties of these scales including their relations to each other.
The second subsection in this section is devoted to the study of
a particular atomic decomposition of the   Besov-type spaces, which is adapted to the radial situation
and which serves as the main tool by dealing with the smoothness and decay properties of radial functions
belonging to  a Besov-type space.


\subsection{The definition and basic properties}\label{bp}


\hskip\parindent
We introduce the scale $\bt$ by using  smooth dyadic decompositions of unity.
Let $\Phi,$ $\vz\in\mathcal{S}(\rn)$ such that
\begin{equation}\label{e1.0}
\supp \widehat{\Phi}\subset \{\xi\in\rn:\,|\xi|\le2\}\quad\mathrm{and}\quad
|\widehat{\Phi}(\xi)|\ge C>0\hs\mathrm{if}\hs |\xi|\le 5/3
\end{equation}
and
\begin{equation}\label{e1.1}
\supp \widehat{\vz}\subset \{\xi\in\rn:\,1/2\le|\xi|\le2\}\quad\mathrm{and}\quad
|\widehat{\vz}(\xi)|\ge C>0\hs\mathrm{if}\hs 3/5\le|\xi|\le 5/3.
\end{equation}

Now we recall the notion of Besov-type spaces $\bt$; see \cite[Definition 2.1]{ysy}.

\begin{definition}\rm\label{d1}
Let $s\in\rr$, $\tau\in[0,\infty)$, $p, q \in(0,\fz]$ and $\Phi$,
$\vz\in\cs(\rn)$ be as in
\eqref{e1.0} and \eqref{e1.1}, respectively.
The \emph{Besov-type space} $\bt$ is defined to be the collection of all $f\in \mathcal{S}'(\rn)$ such that
$$\|f\|_{\bt}:=
\sup_{P\in\mathcal{Q}}\frac1{|P|^{\tau}}\left\{\sum_{j=\max\{j_P,0\}}^\fz
2^{js q}\left[\int_P
|\vz_j\ast f(x)|^p\,dx\right]^{q/p}\right\}^{1/q}<\fz$$
with the usual modifications made in case $p=\fz$ and/or $q=\fz$,
and with $\vz_0$  replaced by $\Phi$ if $j=0$.
\end{definition}

\begin{remark}\label{grund}
(i) The \emph{Besov-type space} $\bt$ is a complete quasi-normed space,
i.\,e., a quasi-Banach space (see \cite[Lemma~2.1]{ysy}).

(ii) In case $\tau =0$ we are back to the standard Besov spaces, i.\,e.,
$B^{s,0}_{p,q}(\rn) = B^{s}_{p,q}(\rn)$.

(iii) We have monotonicity with respect to $s$ and with respect to $q$, i.\,e.,
\[
B^{s_0,\tau}_{p,q_0}(\rn)\hookrightarrow B^{s_1,\tau}_{p,q_1}(\rn)\qquad \mbox{if}\quad s_0 >s_1
\quad \mbox{and} \quad 0 < q_0,q_1 \le \infty\, ,
\]
as well as $B^{s,\tau}_{p,q_0} (\rn)\hookrightarrow B^{s,\tau}_{p,q_1}(\rn)$
if $q_0 \le q_1\, .$

(iv) Let  $s\in\rr$ and $p\in (0,\fz]$. Then it holds that
$B^{s,\tau}_{p,q}(\rn) = B^{s+n(\tau-1/p)}_{\fz,\fz}(\rn)$
if either $q\in(0,\fz)$ and $\tau\in(1/p,\fz)$, or
$q=\fz$ and $\tau\in[1/p,\fz)$; see \cite{yy02}.
In case $s+n(\tau-1/p)>0$ the space $B^{s+n(\tau-1/p)}_{\fz,\fz}(\rn)$ is a H\"older-Zygmund space
with a transparent description in terms of differences; see, e.\,g., \cite[Section 2.5.7]{t83}.

(v) Of course, the spaces defined above, are complicated.
The definition is not really transparent. For this reason the authors
have studied in \cite[Section 4.3]{ysy}
also  characterizations  in terms of differences.

(vi)
The Besov-type space $\bt$ and its homogeneous counterpart
$\dbt$, restricted to the Banach space case, were first introduced by El
Baraka in \cite{el021, el022, el062}.
The extension to quasi-Banach spaces has been done in \cite{yy1,yy2}; see also
\cite{yy102,lsuyy1}.
For a  first systematic study we refer to the lecture note \cite{ysy}.
\end{remark}

Of particular importance for us is the following
embedding into the \emph{class
of all complex-valued, uniformly continuous and bounded functions
on $\rn$}, here denoted by $C_{ub} (\rn)$.

\begin{proposition}\label{grundp}
Let $s\in\rr$, $\tau\in[0,\infty)$ and  $p, q \in(0,\fz]$

${\rm(i)}$ If $s+n(\tau-1/p)>0$, then $B^{s,\tau}_{p,q}(\rn) \hookrightarrow C_{ub} (\rn)$\,.

${\rm(ii)}$  Let $p\in(0,\fz)$, $q\in(0,\fz]$, $\tau\in(0, 1/p)$ and
$s+ n\tau -\frac np =0$.
Then $
B^{s,\tau}_{p,q}(\rn) \not \subset C_{ub}(\rn) \, .$

${\rm(iii)}$ Let $p\in(0,\fz)$ and $q\in(0,\fz]$.
Then $B^{0, 1/p}_{p,q}(\rn)  \not \subset C_{ub}(\rn)\, .$
\end{proposition}

\begin{remark}
For a proof of Proposition \ref{grundp}, we refer to \cite[Proposition 2.6(i)]{ysy}
and to \cite{s011}.
\end{remark}

For the convenience of the reader we also recall the definition of Besov-Morrey
spaces as follows.

\begin{definition}\label{yamazaki}
Let $\Phi$ and $\vz\in\cs(\rn)$ be as in \eqref{e1.0} and \eqref{e1.1}, respectively.
Let $s\in\rr$, $ 0 <  p \le u \le \infty $ and $0 < q \le \infty$.
Then the \emph{Besov-Morrey space} $\cn^s_{u,p,q}(\rn)$ is defined to be the set of all $f\in\cs'(\rn)$ such that
\begin{equation*}
\|f\|_{\cn^s_{u,p,q}(\rn)}:= \lf[\sum^{\infty}_{j= 0}
 2^{jsq}\,   \| \, \varphi_j \ast  f \, \|_{\cm^u_p(\rn)}^q \r]^{1/q} < \infty
\end{equation*}
with the usual modifications made in case $p=\fz$ and/or $q=\fz$, and with $\vz_0$
replaced by $\Phi$ if $j=0$.
\end{definition}

\begin{remark}\label{gleich3}
(i) The Besov-Morrey spaces $\cn^s_{u,p,q}(\rn)$ represent the Besov scale built on the Morrey space $\cm^u_{p}(\rn)$.
Kozono and Yamazaki \cite{KY} in 1994 and  later on Mazzucato \cite{ma03}
have been the first who investigated spaces of this type.
In fact, they studied two, slightly different, types of spaces.
The first modification consists in restricting the supremum within the definition of the Morrey norm to balls
with volume $\le 1$. Secondly, they studied homogeneous Besov-Morrey spaces $\dot{\cn}^s_{u,p,q}(\rn)$.
For more information on these spaces, we
refer to \cite{KY,ma03,tx,st} and the survey \cite{s011}.

(ii) A further family of relatives of $\bt$ and $\cn^s_{u,p,q}(\rn)$ has been  introduced and investigated by Triebel
in his recent book \cite{t12}.
\end{remark}

\begin{lemma}\label{emb1}
Let $s\in \rr$ and  $0 <  p \le u \le \infty$.

${\rm(i)}$ If $q\in (0,\infty)$, then $\cn^s_{u,p,q}(\rn)   \hookrightarrow
B^{s,\frac 1p- \frac 1u}_{p,q}(\rn)\, $.

${\rm(ii)}$ It also holds that $
\cn^s_{u,p,\infty} (\rn)=  B^{s,\frac 1p- \frac 1u}_{p,\infty}(\rn)$
in the sense of equivalent quasi-norms.

${\rm(iii)}$ Let $m \in \nn_0$ and $1 \le p \le u <\infty$. Then
\[
\cn^m_{u,p,\min(2,p)} (\rn) \hookrightarrow W^m \cm^u_p(\rn)
\hookrightarrow   \cn^{m}_{u, p,\infty} (\rn)\, .
\]
\end{lemma}

\begin{remark}
Lemma \ref{emb1}(i) and (ii) have been proved in \cite{syy}; see also \cite[Proposition~7]{s011}.
There the authors argued with atomic decompositions.
In addition they have been able to show that the embedding in Lemma \ref{emb1}(i) is proper if $p < u$.
Concerning part (iii) of Lemma \ref{emb1}
we refer to Sawano \cite{sa0} and \cite[Lemma~5]{s011}.
\end{remark}


\subsection{Radial subspaces}\label{rs}


\hskip\parindent
Let $U$ be an isometry of $\rn$ and   $g\in\cs(\rn)$.
Then we define $g^U(x):=g(Ux)$,  $x\in\rn$. For $f\in\cs'(\rn)$ we put
\[
f^U(g):=f(g^{U^{-1}})\, , \qquad g\in\cs(\rn)\, ,
\]
where $U^{-1}$
denotes the isometry inverse to $U$. Then $f^U$ is also a distribution in $\cs'(\rn)$.

Let $SO(\rn)$ be the \emph{group of rotations around the origin in $\rn$}.
We say that $f\in\cs'(\rn)$ is \emph{invariant with respect to $SO(\rn)$}
if $f^U=f$ for all $U\in SO(\rn)$. Now we are able to define the radial subspaces
which we are interested in.

\begin{definition}\label{d3.1}
Let $s,\,\tau,\,p$ and $q$ be as in Definition \ref{d1}. The \emph{radial subspace}
$R\bt$ of the space $\bt$ is defined by
$$R\bt:=\lf\{f\in \bt:\ f\ \mathrm{is}\ \mathrm{invariant}\ \mathrm{with}\
\mathrm{respect}\ \mathrm{to}\ SO(\rn)\r\}.$$
\end{definition}

Since the linear operator $T_Uf:=f^U-f$
is bounded on $\bt$ for all $U\in SO(\rn)$
and
\[
R\bt = \bigcap_{U\in SO(\rn)} \mathrm{Ker}\,T_U \, ,
\]
we know that $R\bt$ is a closed subspace of $\bt$,
and hence a quasi-Banach space with respect to the induced quasi-norm.


\subsection{Atomic decompositions}\label{ado}


\hskip\parindent
As mentioned above our main tool in all our investigations
is the description of these radial subspaces $R\bt$
by means of atoms.
Starting point is a rather general characterization of $\bt$ in terms of atoms.
Since this is, in a certain sense,  parallel to  what has been done in
Frazier and Jawerth \cite{fj85,fj90} and also \cite{ss00}, we shifted this more
technical part to the Appendix at the end of this paper.

Atomic decompositions are always connected with a sequence of coverings of $\rn$.
Here we are interested in a very special sequence adapted to the radial situation.
To begin with we recall some notions; see
Skrzypczak in \cite{s98}, but also \cite[Definition 1]{ss00}.
The  notion of an atom, we are using here, represents a certain modification of the
definition given  in Frazier and Jawerth \cite{fj85,fj90}.
For any open set $Q\subset\rn$ and any  $r\in(0,\fz)$, we put $rQ:=\{x\in\rn:\ \dist(x,Q)<r\}.$
Also we use the abbreviation $p':= p/(p-1)$ if
$1 < p \le \infty$ and $p':= \infty $ if $0 < p \le 1$.

\begin{definition}\label{d2.1}
Let $s\in\rr$, $p\in(0,\fz]$, $r\in(0,\fz)$, $L$ and $M$ be integers such that $L\ge0$ and
$M\ge-1$. Assume that $Q\subset \rn$ is an open connected set with $\mathrm{diam}\,Q=r$.

{\rm(i)} A smooth function $a$ is called an \emph{$1_L$-atom centered in $Q$}
if  $\supp a \subset \frac r2Q$ and
$$\sup_{y\in\rn}|\partial^\az a(y)|\le 1\quad \mathrm{for}\quad |\az|\le L.$$

{\rm(ii)} A smooth function $a$ is called an
\emph{$(s,p)_{L,M}$-atom centered in $Q$} if
$\supp a \subset \frac r2Q$,
$$\sup_{y\in\rn}|\partial^\az a(y)|\le r^{s-|\az|-\frac np}
\quad \mathrm{for}\quad |\az|\le L$$
and
$$\left|\int_\rn a(y)\vz(y)\,dy\right|\le r^{s+M+1+n/p'}\|\vz\|_{C^{M+1}(\overline{rQ})}$$
for all $\vz\in C^\fz(\rn)$, where $\|\vz\|_{C^{M+1}(\overline{rQ})}
:=\sup_{x\in\overline{rQ}}\sup_{|\az|\le M+1} |\partial^\az\vz(x)|$.
\end{definition}

\begin{remark}
We recall that when $M=-1$, then the second  condition in
 Definition \ref{d2.1}(ii) is void. Of certain use is the following simple observation:
if $a$ is an $(s_0,p)_{L,M}$-atom, then $r^{s_1-s_0}\, a$ is also an $(s_1,p)_{L,M}$-atom.
\end{remark}

\subsection*{Regular coverings}

\hskip\parindent
We follow \cite{ss00}.
Let $\{Q_\ell\}_{\ell}$ be a covering of $\rn$ by connected open sets $Q_\ell$.
We put
\[
C_Q := \sup_{x\in \rn} |\{\ell : \quad Q_\ell \ \ \mbox{contains}\ \ x\}|
\]
(here $|\, \cdot \, |$ denotes the \emph{cardinality} of the set).
We call this number the \emph{multiplicity of the covering $\{Q_\ell\}_{\ell}$}.
A covering with finite multiplicity is called \emph{uniformly locally finite}.
Let $(\Omega_j:=\{\Omega_{j,\ell}\}_{\ell=0}^\fz)_{j=0}^\fz$ be a sequence
of uniformly locally finite coverings of $\rn$. The supremum of multiplicities of
the coverings $\Omega_j$, $j\in\nn_0,$ is called
the \emph{multiplicity of the sequence  $(\Omega_j)_{j=0}^\fz$}.

\begin{definition}\label{d2.2}
A sequence $(\Omega_j)_{j=0}^\fz:=
(\{\Omega_{j,\ell}\}_{\ell=0}^\fz)_{j=0}^\fz$ of coverings of $\rn$
is \emph{regular} if the following conditions are satisfied:

(i) $\rn\subset \cup_{\ell\in\nn_0} \overline{\Omega_{j,\ell}}$ for all
$j\in \nn_0;$

(ii) there exists some positive number $\varepsilon_0$ such that for all
$\varepsilon\in(0,\varepsilon_0)$, the sequences of coverings $\{\varepsilon2^{-j}\Omega_{j,\ell}\}_{\ell=0}^\fz$ have
finite multiplicity with uniform bound of multiplicity with respect to $\varepsilon;$

(iii) there exist positive numbers $B_n$ and $C_n$, depending
only on the dimension $n$, such that
$\mathrm{diam}\,\Omega_{j,\ell}\le B_n 2^{-j}$ and $C_n2^{-jn}\le |\Omega_{j,\ell}|.$
\end{definition}

\begin{remark}\label{Konstanten}
 \rm
(i) By $\omega_n$ we denote the \emph{volume of the unit ball
in $\rn$}. Let $A_n := (C_n/\omega_n)^{1/n}$.
Then Definition \ref{d2.2}(iii) implies that
$A_n \, 2^{-j} \le \diam \Omega_{j, \ell} \le B_n \,  2^{-j}$
and
\[
C_n \, 2^{-jn}\le
|\Omega_{j, \ell}| \le B_n^n \,\omega_n \, 2^{-jn}
\]
for all $j$ and all $\ell$.

(ii)
If $(\{\Omega_{j,\ell}\}_{\ell=0}^\fz)_{j=0}^\fz$ is a regular sequence
of coverings, the cardinality of the sets,
$$I_{j,\ell}:=\{k:\ \Omega_{j,\ell}\cap \Omega_{j,k}\neq \emptyset\}\, , \qquad j,\, k \in \nn_0\, , $$
is uniformly bounded; see \cite[Lemma 1]{ss00}.
\end{remark}

Next we  define some  sequence spaces with respect to a regular sequence $(\{\Omega_{j,\ell}\}_{\ell=0}^\fz)_{j=0}^\fz$
of coverings.

\begin{definition}\label{d2.3}
Let $\tau\in[0,\fz)$ and  $p, q\in(0,\fz]$. Let
$\Omega:=(\{\Omega_{j,\ell}\}_{\ell=0}^\fz)_{j=0}^\fz$ be a regular sequence
of coverings.
Then the sequence space $b(p,q,\tau,\Omega)$ is defined to
be the set of all sequences $t:=\{t_{j,\ell}\}_{j,\ell=0}^\fz\subset\cc$ such that
\begin{eqnarray*}
\|t\|_{b(p,q,\tau,\Omega)}:=\sup_{P\in\mathcal{Q}}\frac1{|P|^\tau}
\left\{\sum_{j=\max\{j_P,0\}}^\fz \left[\sum_{\gfz{\ell\in\nn_0}{\Omega_{j,\ell}\cap P
\neq \emptyset}}|t_{j,\ell}|^p\right]^{q/p}\right\}^{1/q}<\fz.
\end{eqnarray*}
\end{definition}

\begin{remark}
 \rm
In case $\tau<1/p$, the replacement of  the sum  $\sum_{j=\max\{j_P,0\}}^\fz$
in Definition \ref{d2.3} by $\sum_{j=0}^\fz$ yields an equivalent quasi-norm.
For a proof  we refer to \cite{syy}.
\end{remark}

\subsection*{Regular coverings and associated atomic decompositions}

\hskip\parindent
The spaces $\bt$ are said to admit an atomic decomposition with respect to the sequence
$(\{\Omega_{j,\ell}\}_{\ell=0}^\fz)_{j=0}^\fz$ if there exist integers $L$ and $M$ such that

(i) any $f\in\bt$ can be represented by
\begin{eqnarray}\label{e2.1}
f=\sum_{j=0}^\fz\sum_{\ell=0}^\fz t_{j,\ell}\,a_{j,\ell}
\end{eqnarray}
in $\cs'(\rn)$, where
$a_{0,\ell}$ is an $1_L$-atom centered at the set $\Omega_{0,\ell}$,
$a_{j,\ell}$ with $j\in\nn$ is an $(s,p)_{L,M}$-atom centered at the set $\Omega_{j,\ell}$, and
$\{t_{j,\ell}\}_{j,\ell\in\nn_0}$ is a sequence of complex numbers
satisfying that
\begin{eqnarray}\label{e2.2}
\|\{t_{j,\ell}\}_{j,\ell\in\nn_0}\|_{b(p,q,\tau,\Omega)}<\fz;
\end{eqnarray}

(ii) any $f\in\cs'(\rn)$ given by \eqref{e2.1} with \eqref{e2.2}
is an element of $\bt$;

(iii) the infimum of \eqref{e2.2} with respect to all admissible
representations of $f$ is an equivalent norm in $\bt$.

Next we show that if $\{\Omega_{j}\}_{j=0}^\fz$ is
a regular sequence of coverings, then for some
 $L$ and $M$, the spaces $\bt$ admit an atomic decomposition with
respect to $\{\Omega_{j}\}_{j=0}^\fz$. In what follows, the \emph{symbol}
$\lfloor \az\rfloor$ for any $\az\in\rr$
denotes the maximal integer not more than $\az$.

\begin{theorem}\label{t2.1}
Let $s\in\rr$, $\tau\in[0,\fz)$ and $p,\,q\in(0,\fz]$. Let  $\{\Omega_{j}\}_{j=0}^\fz$ be
a regular sequence of coverings and let $L$ and $M$ be integers such that
\[
L\ge \max\{\lfloor s+n\tau\rfloor+1,0\} \qquad \mbox{and}\qquad
M\ge \max\{\lfloor \sigma_{p}-s\rfloor,-1\} \,,
\]
where $\sz_p:=n\,(\max\{0, 1/p-1\})$.
Then the spaces $\bt$ admit an atomic decomposition with
respect to $\{\Omega_{j}\}_{j=0}^\fz$.
\end{theorem}

\begin{remark}
This theorem represents an extension of the well-known
characterizations by atoms discussed in Frazier and Jawerth
\cite{fj85,fj90}. More exactly, choosing
\[
\Omega_j := \{Q_{j,k}: \quad k \in \zz^n\}\, , \qquad j \in \nn_0\,
\]
(see (\ref{dyadic})) and $\tau =0$, then Theorem \ref{t2.1} has been proved in the
quoted papers by Frazier and Jawerth.
Regular coverings and related atomic decompositions, also restricted to $\tau =0$,
have been discussed in  \cite{ss00}; see also \cite{s02}.
Atomic decompositions of Besov-type spaces have been investigated in
\cite{ysy} and \cite{lsuyy}.
Our proof of Theorem \ref{t2.1} is more or
less parallel to those given in the quoted papers.
For that reason we shift it into the Appendix at the end of this paper.
\end{remark}


\subsection{Atomic decompositions of radial subspaces}\label{adors}


\hskip\parindent
Now we make use of the flexibility in the choice of the regular coverings.
In \cite[Section 3.2]{ss00} Sickel and Skrzypczak constructed a sequence of
regular coverings which is well adapted to the radial situation.

A basic role in this construction is played by
the following shells (balls if $k=0$):
\[
P_{j,k}:= \lf\{ x \in \rn \, : \quad k \, 2^{-j} \le |x| <
(k+1)\, 2^{-j}\r\}\, , \qquad j,\,k\in\nn_0\, .
\]

\begin{lemma}\label{l3.1}
Let $n\ge 2$.
There exists a regular sequence
$$(\Omega^R_j)_{j\in\nn_0}:=\lf(\{\Omega^R_{j,k,\ell}\}_{k\in\nn_0,
\ell\in\{1,\ldots,C(n,k)\}}\r)_{j\in\nn_0}$$ of coverings with finite multiplicity satisfying that

\rm{(a)}
all $\Omega^R_{j,k,\ell}$ are balls  with centers $y_{j,k,\ell}$ satisfying that
\[
|y_{j,k,\ell}| = \left\{\begin{array}{lll}
2^{-j}\, (k+1/2) & \qquad & \mbox{if}\quad k\in\nn\, , \\
0 && \mbox{if}\quad k=0\, ;
\end{array}\right.
\]

\rm{(b)}
$$P_{j,k} \subset  \bigcup_{\ell=1}^{C(n,k)}
\Omega^R_{j,k,\ell}\, , \qquad j\in\nn_0\, ; $$

\rm{(c)}
$\diam \Omega^R_{j,k,\ell} = 12 \, \cdot \,  2^{-j}\,  ; $

\rm{(d)} the sums 
$$\sum_{k=0}^\infty \sum_{\ell=1}^{C(n,k)}
\chi_{j,k,\ell}(x)$$  
are uniformly bounded in $x\in \rn$ and
$j\in\nn_0 $ (here $\chi_{j,k,\ell}$ denotes the
characteristic function of $\Omega^R_{j,k,\ell}$);

\rm{(e)}
$$\Omega^R_{j,k,\ell}= \{x\in \rn\, : \: 2^j\, x \in
\Omega^R_{0,k,\ell}\}, \, \, j\in\nn_0\,; $$

\rm{(f)} $\, \, C(n,0)=1$, $C(n,k)\in\nn$ and $C(n,k)\asymp k^{n-1}$,  when $k\in\nn$;

\rm{(g)}
with an appropriate enumeration it holds that
\[
\{(x_1,0,\ldots , 0)\, : \quad x_1 \ge 0\} \subset
\bigcup_{k=0}^\infty \Omega^R_{j,k,1}
\]
and
\[
\lf|\, \{(x_1,0,\ldots , 0)\, : \quad x_1 \in \rr\} \cap
\Omega^R_{j,k,1}\, \r| \: \asymp \: 2^{-j}\, .
\]

\rm{(h)} for all $k\in\nn$ and $\ell\in\{1,\ldots,C(n,k)\}$ there exists an element $U_{k,\ell}\in SO(\rn)$
such that $\Omega^R_{j,k,\ell}=U_{k,\ell}(\Omega^R_{j,k,1});$

\rm{(i)} Let $\varepsilon\in(0,\varepsilon_0)$. Then the multiplicity of the sequence
$$\lf(\lf\{\varepsilon\Omega^R_{j,k,\ell}:\ \ell\in\{1,\ldots, C(n,k)\}, k\in\nn_0\r\}\r)_{j\in\nn_0}$$
is finite with multiplicity constant independent of $\varepsilon$.
\end{lemma}

If $k=0$, we sometimes use the notation $\Omega^R_{j,0}$ to replace $\Omega^R_{j,0,1}$.
Lemma \ref{l3.1}(g) is not stated explicitly in \cite{ss00}
but it follows immediately from the construction described there; see also \cite{KLSS0}.
Most transparent is the case $n=2$. Here is a picture.

\begin{center}
\includegraphics[height=9cm]{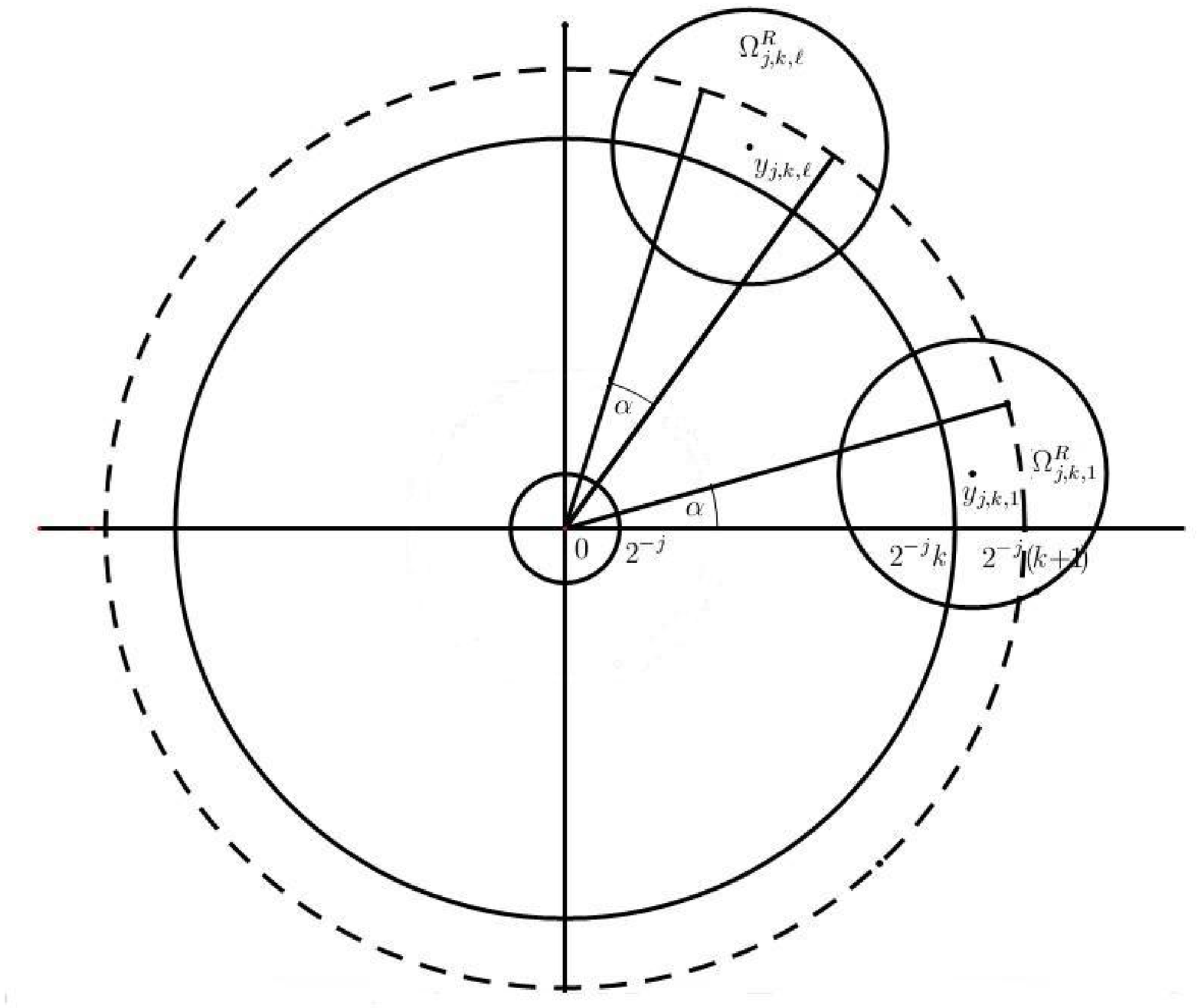}

\small{Fig.1\hs\hs  A piece of the covering $(\Omega^R_{j,k,\ell})_{j=0}^\fz$
in the case $n=2$}
\end{center}

\noindent
Here the angle $\alpha$ is taken to be $(2\pi)/(2k+1)$, $k \in\nn$.

Applying Theorem \ref{t2.1} and Lemma \ref{l3.1}, we obtain the atomic
decomposition of radial spaces.
We use the following abbreviation
\[
\omega (P,j,k) :=|\{\ell: \quad 1\le \ell \le C(n,k)\quad \mbox{and}
\quad \Omega^R_{j,k,\ell}\cap P\neq \emptyset \}|\, .
\]
Obviously, $0 \le \omega (P,j,0) \le 1$ for all $j$ and all dyadic cubes $P$.

\begin{theorem}\label{t3.1}
Let $n\ge 2$, $s\in\rr$, $\tau\in[0,\fz)$ and $p,\,q\in(0,\fz]$, integers $L,M$ satisfying
\[
L\ge \max\{\lfloor s+n\tau\rfloor+1,0\}
\qquad \mbox{and} \qquad
M\ge \max\{\lfloor \sigma_{p}-s\rfloor,-1\} \, ,
\]
and $(\Omega^R_j)_{j\in\nn_0}:=(\{\Omega^R_{j,k,\ell}\}_{k\in\nn_0,
\ell\in\{1,\ldots,C(n,k)\}})_{j\in\nn_0}$ be as in Lemma \ref{l3.1}.

{\rm (i)} The space $R\bt$ admits an atomic decomposition with
respect to $(\Omega^R_{j})_{j\in\nn_0}$.

{\rm (ii)} If $f\in R\bt$, then $f$ admits an atomic  decomposition
\begin{equation}\label{deco}
f = \sum_{j=0}^\infty t_{j,0}\, a_{j,0}+\sum_{j=0}^\infty \sum_{k=1}^\infty  t_{j,k}
\, \sum_{\ell=1}^{C(n,k)}\,  a_{j,k,\ell} \, ,
\end{equation}
where $a_{j,k,\ell}$ is an $(s,p)_{L,M}$-atom with respect to $\Omega^R_{j,k,\ell}$ for
all $j\in\nn$,  $a_{j,0}$ an $(s,p)_{L,M}$-atom with respect to $\Omega^R_{j,0}$ for
all $j\in\nn$, $a_{0,k,\ell}$ an $1_L$-atom with respect to $\Omega^R_{0,k,\ell}$,
$a_{0,0}$ an $1_L$-atom with respect to $\Omega^R_{0,0}$ and
\begin{equation}\label{deco2}
\sup_{P\in\mathcal{Q}}\frac1{|P|^\tau}
\left\{\sum_{j=\max\{j_P,0\}}^\fz
\left[\sum_{k=0}^\infty\, \omega (P,j,k)\,
|t_{j,k}|^p\right]^{\frac qp}\right\}^{\frac 1q}\le C\|f\|_{\bt} \, .
\end{equation}
Here $C$ is a positive constant independent of $f$.
\end{theorem}

\begin{proof}
Part (i) is an immediate consequence of Theorem \ref{t2.1}
and Lemma \ref{l3.1}.
Concerning (ii) we employ the particular decomposition constructed in
Step 1 of the proof of Theorem \ref{t2.1}.
Rewriting (\ref{e2.3}) we obtain
\[
t_{j,k,\ell}:= \left\{ \begin{array}{lll}
D(n,M)\,  \dsup_{y \in \Omega^R_{0,k,\ell} }\, | \Phi \ast f(y)|,
& \qquad & \mbox{if}\quad j=0\, ;\\
&&\\
E(n,M)\, 2^ {j(s-n/p)}\, \dsup_{y \in \Omega^R_{j,k,\ell} }\, | \varphi_j \ast f(y)|,
& \qquad & \mbox{if}\quad j \in \nn\, .
\end{array}\right.
\]
Since $\Phi\ast f$ and $\varphi_j\ast f$ are radial functions if $f$, $\Phi$
and $\vz$ are radial, we
conclude by using Lemma \ref{l3.1}(h) of our regular covering that
\[
t_{j,k,\ell} = t_{j,k,1} \qquad \mbox{for all}\quad j,k, \ell \, .
\]
This proves (ii) and hence finishes the proof of Theorem \ref{t3.1}.
\end{proof}

\begin{remark}
Epperson and Frazier \cite{ef95,ef96} independently developed a theory of radial
subspaces of Besov and Triebel-Lizorkin spaces. They preferred to work with atomic
decompositions, where the atoms themselves are radial. We did not follow their
treatment here but sometimes this different point of view has some advantages.
\end{remark}


\section{On the smoothness of radial functions outside the origin}
\label{smooth}


\hskip\parindent
To begin with we study the continuity of radial functions outside the origin.
As usual, we  call an element  $f$ of $\bt$
continuous if it has a continuous representative.

Before we start to investigate this problem we  explain how we  use
the atomic decomposition in (\ref{deco}).
This is also of some importance for all other proofs here.
The main feature is the local structure in the following sense.
Let $f$ be given by (\ref{deco}), i.\,e.,
\[
f = \sum_{j=0}^\infty t_{j,0}\, a_{j,0}+\sum_{j=0}^\infty \sum_{k=1}^\infty
t_{j,k} \, \sum_{\ell=1}^{C(n,k)}\,  a_{j,k,\ell} \, ,
\]
We fix $|x|\ge 1$.
Observe that for all $j\in\nn_0$ there exists $k_j\in \nn$ such that
\begin{equation}\label{e1}
k_j \, 2^{-j} \le |x| < (k_j + 1)\, 2^{-j} \, .
\end{equation}
Then, employing the restrictions on the supports of our atoms, we see that
$f(x)$ is given by
\[
\sum_{|r_j|\le N} \sum_{|t_j|\le N} \sum_{j=0}^\infty  \, t_{j,k_j+r_j}\,
a_{j,k_j+r_j, \ell_j + t_j}(x) \, ,
\]
where $N \in \nn$ is a universal number independent of $f$  and $x$.
Here $\ell_j$ depends on $x$.
To avoid technicalities
we simply deal in our investigations with the function
\begin{equation}\label{point1}
f^P (x) = \sum_{j=0}^\infty  \, s_{j,k_j}\,
a_{j,k_j, \ell_j}(x)
\end{equation}
instead of $f$ itself.
In case $0 < |x|<1$ these arguments have to be modified. Let
\begin{equation}\label{e2}
2^{-j_0} \le |x| <  2^{-j_0+1}
\end{equation}
for some $j_0 \in \nn$.
Then we  deal  with the function
\begin{equation}\label{point2}
f^P (x) := \sum_{j=0}^{j_0 -1}  \, s_{j,0}\, a_{j,0}(x)  + \sum_{j=j_0}^\infty  \, s_{j,k_j}\,
a_{j,k_j, \ell_j}(x) \, ,
\end{equation}
as a replacement of $f$ at
the point $x$. Here $k_j$ is defined as in (\ref{e1}).

\begin{theorem}\label{tc.1}
Let $n \ge 2$, $p\in(0,\fz)$ and $\tau \in [0,\fz)$.
Let either $s>1/p$ and $q\in(0,\fz]$, or $s=1/p$ and $q\in(0,1]$.
Every radial element $f $ in $\bt$
has a representative $\widetilde{f}$ which is continuous outside of the origin.
\end{theorem}

\begin{proof}
{\em Step 1.} Let $\tau \ge 1/p$.
Then by \cite[Proposition 2.6]{ysy}, we know that
\[
\bt \hookrightarrow B^s_{\infty,\infty} (\rn) \qquad \mbox{if}\quad s>0\,.
\]
Hence, all functions in $\bt$ are continuous on $\rn$, not only the radial functions.

{\em Step 2.} Suppose $0 \le \tau < 1/p$.
By monotonicity of the spaces $\bt$ (see Remark \ref{grund}),  it is enough to deal with
$B^{\frac 1p, \, \tau}_{p,1} (\rn)$.

{\em Substep 2.1.} Let $|x|\ge 1$. We investigate the sequence
\[
S_N f^P (x) := \sum_{j=0}^N  \, s_{j,k_j}\,
a_{j,k_j, \ell_j}(x)\, , \qquad N\in\nn .
\]
Since our atoms are at least continuous, also $S_N f^P $ is continuous.
Furthermore, employing the normalization of our atoms
(see Definition  \ref{d2.1}), we find that
\begin{eqnarray*}
|S_N f^P(x)|\le \sum_{j=0}^N 2^{-j(s-n/p)}\, |t_{j,k_j}|=
\sum_{j=0}^N (k_j)^{\frac{1-n}p}\, 2^{-j(s-n/p)} \, (k_j^{n-1}\, |t_{j,k_j}|^p )^{1/p}.
\end{eqnarray*}
Let $P$ be the smallest dyadic cube containing all points
$ (k_j + 1)\, 2^{-j} \, x/|x|$, $j \in \nn_0$,  and $0$.
Hence
\[
|P|\asymp |x|^n \asymp (k_j+1)^n 2^{-jn} \qquad  \mbox{and} \qquad \omega(P,j,k_j) \asymp C(n,k_j) \asymp k_j^{n-1}\, ,
\]
since $k_j \in\nn$. We add a short explanation for the last relation.
Take $x$ such that all components are nonnegative.
Let $2^L < (k_j + 1)\, 2^{-j} \le 2^{L+1}$. Then $P = Q_{-L, 0}$ (see (\ref{dyadic})),
and contains the part of the shell (\ref{e1})
which is contained in
\[
\{y= (y_1, \ldots y_n): \: y_i \ge 0 \: \mbox{ for all }\:i\}\, .
\]
By the uniform distribution of the balls $\Omega^R_{j,k_j,\ell}$ on the shell this implies
that
\[
|\{ \ell : \quad P \cap \Omega^R_{j,k_j,\ell}\}| \asymp \frac{C(n,k_j)}{2^n}\, .
\]
Since we ignore the dependence on $n$ the claim follows.
This results in
\begin{eqnarray}\label{conv1}
|S_N f^P (x)|&\ls &
\sum_{j=0}^N (k_j)^{\frac{1-n}p}\, 2^{-j(s-n/p)} \lf[
\sum_{\ell=1}^\infty \omega (P,j,\ell ) \, |t_{j,\ell}|^p\r]^{1/p}\\
& \ls & |x|^{n\tau} \, \frac{|x|^{\frac{1-n}p}}{|P|^\tau}
\sum_{j=0}^\fz \, 2^{-j(s-1/p)} \, \lf[
\sum_{\ell=1}^\infty \omega (P,j,\ell ) \, |t_{j,\ell}|^p\r]^{1/p}
\nonumber\\
&\ls & |x|^{\frac{1-n}p+n\tau} \|f\|_{B^{\frac 1p, \, \tau}_{p,1} (\rn)}\, ,\nonumber
\end{eqnarray}
where the last inequality follows from (\ref{deco2}). Here the constant behind $\ls$ does not depend on $f$ and $x$.
 Practically the same set of inequalities as in (\ref{conv1}) yields
\[
 |S_{N_1+ N_2} f^P (x) - S_{N_1} f^P (x)| \ls
\frac{|x|^{\frac{1-n}p}}{|P|^\tau}
\sum_{j=N_1}^\fz \, 2^{-j(s-1/p)} \, \lf[
\sum_{\ell=1}^\infty \omega (P,j,\ell ) \, |t_{j,\ell}|^p\r]^{1/p} \ls \varepsilon,
\]
if $N_1$ is sufficiently large (depending on $\varepsilon$ and $|x|$) and $N_2 \in \nn$ is arbitrary. Let $b>1$.
Then this  implies uniform convergence of the sequence $S_N f^P$
on any shell $\{x \in \rn: \: 1 < |x|\le b\}$ to a continuous limit.
Coming back to the original situation we obtain the
continuity of $f$ on these regions,
which means on $|x|\ge 1$.

{\em Substep 2.2.} Let $|x|< 1$.
The arguments are similar with a few modifications.
There exists a natural number $j_0$ such that
$2^{-j_0}\le |x|<2^{-j_0+1}.$

We work with
\[
S_N f^P(x):= \sum_{j=0}^{j_0-1}\, t_{j,0}\, a_{j,0}(x)+
\sum_{j=j_0}^N t_{j,k_j}\, a _{j,k_j,\ell_j}(x)\, , \qquad N \ge j_0\, ,
\]
instead of $f$ itself. As in the previous substep we obtain
\[
|S_N f^P(x)|\le \sum_{j=0}^{j_0-1}2^{-j(s-n/p)}\, |t_{j,0}|+
\sum_{j=j_0}^N 2^{-j(s-n/p)}\, |t_{j,k_j}|.
\]
To estimate the second sum on the right-hand side we argue as in Substep 2.1. Indeed,
let $P$ be the smallest dyadic cube containing
$ (k_j + 1)\, 2^{-j} \, x/|x|$, $j \in \nn_0$,
$2^{-j_0 +1}x/|x|$ and $0$.
Then
\[
|P|\asymp  |x|^n \asymp  2^{-j_0n} \qquad \mbox{and}\qquad
\omega(P,j,k_j) \asymp C(n,k_j) \asymp k_j^{n-1}\, , \quad j \ge j_0\, .
\]
Thus, repeating the argument in Substep 2.1, we conclude that
\begin{eqnarray*}
\sum_{j=j_0}^N 2^{-j(s-n/p)}|t_{j,k_j}|\ls |x|^{\frac{1-n}p+n\tau} \|f\|_{B^{1/p,\tau}_{p,1}(\rn)}
\end{eqnarray*}
with a positive constant behind $\ls$ independent of $f$, $N$, $j_0$ and $x$.
Observe that for all $j\in\nn_0$,
$$|t_{j,0}|\ls 2^{-jn\tau}\|f\|_{B^{1/p,\tau}_{p,1}(\rn)}\, ;$$
see (\ref{deco2}). We choose $P$ as above and
use  $ \omega (P,j,0) = 1$.
By means of this inequality, we obtain
\[
\sum_{j=0}^{j_0-1}2^{-j(1/p-n/p)} \, |t_{j,0}|
\ls \|f\|_{B^{1/p,\tau}_{p,1}(\rn)}\sum_{j=0}^{j_0-1}2^{-j(1/p-n/p+n\tau)}\, .
\]
An easy calculation yields
\begin{eqnarray*}
\sum_{j=0}^{j_0-1}2^{-j(1/p-n/p+n\tau)}
&\ls &
\begin{cases} 1, \ \  &\tau\in (\frac{n-1}{np},\frac1p];\\
j_0,\  & \tau=\frac{n-1}{np};\\
2^{-j_0(1/p-n/p+n\tau)}, \ \ & \tau\in[0,\frac{n-1}{np})
 \end{cases}\\
& \asymp &
\begin{cases} 1, \ \  &\tau\in (\frac{n-1}{np},\frac1p];\\
1-\log_2|x|,\  & \tau=\frac{n-1}{np};\\
|x|^{\frac{1-n}p+n\tau}, \ \ & \tau\in[0,\frac{n-1}{np}) \, .
 \end{cases}
\end{eqnarray*}
Summarizing, we have found the estimate
\begin{equation}\label{conv3}
|S_N f^P(x)|\le w(|x|, \tau ) \|f\|_{B^{1/p,\tau}_{p,1}(\rn)}\, ,
\end{equation}
where
\[
w(|x|,\tau):= \begin{cases} 1, \ \  &\tau\in (\frac{n-1}{np},\frac1p];\\
1-\log_2|x|,\  & \tau=\frac{n-1}{np};\\
|x|^{\frac{1-n}p+n\tau}, \ \ & \tau\in[0,\frac{n-1}{np}) \, .
 \end{cases}
\]
Having established this estimate we may proceed as at the end of
Substep 2.1 to convert it into the continuity
of $f$ in all shells (\ref{e2}), $j \in \nn$.
The proof is complete.
\end{proof}

\begin{remark}
(i) For fixed $p$ and $\tau$ the largest space $\bt$ with the indicated property in Theorem \ref{tc.1}
is given by $B^{1/p, \tau}_{p,1} (\rn)$; see Remark \ref{grund} and Proposition \ref{grundp}.

(ii) As mentioned in the introduction this result has many forerunners in case $\tau =0$.
Let us mention Strauss \cite{St} ($H^1(\rn)=W^1_2 (\rn)= B^1_{2,2}(\rn) $), Lions
\cite{Li} ($W^1_p (\rn)$) and Sickel and Skrzypczak \cite{ss00}
(general case with $\tau =0$). A much more detailed analysis of the smoothness of radial
functions belonging to some Besov or Triebel-Lizorkin spaces has been given in
Sickel,  Skrzypczak and Vybiral \cite{ssv}.
In particular it is shown there that, in case $\tau =0$, the results are unimprovable within the scale
$B^s_{p,q} (\rn)$. More exactly, in  $B^{1/p}_{p,q} (\rn)$, $q>1$, there
exist radial and unbounded functions $f$ such that
\[
\supp f \subset \{x \in \rn: \quad a < |x|< b\}, \qquad 0 < a < b < \infty\, .
\]
Here $a$ and $b$ are at our disposal.
\end{remark}


\section{Decay near infinity and controlled unboundedness near the origin of radial functions}
\label{decay}


\hskip\parindent
The proof of Theorem \ref{tc.1} allows some immediate conclusions about the pointwise behaviour of radial functions.
Observe that, under the restrictions of Theorem \ref{tc.1}, our radial functions in
$\bt$ are continuous outside the origin,
hence pointwise inequalities make sense.

\begin{theorem}\label{t3.2}
Let $n\ge 2$ and $p\in(0,\fz)$.

{\rm (i)} Let  $\tau \in (\frac{n-1}{np}, \fz)$, $s+n(\tau-1/p)>0$ and $q\in(0,\fz]$.
Then there exists a positive constant $C$ such that for all
$f\in R\bt$ and all $x \in \rn$, it holds that
\begin{equation}
\label{eq-erg1c}
|f(x)|\le C \, \|f\|_{\bt}\,  .
\end{equation}

{\rm (ii)} Let $\tau \in [0,\frac{n-1}{np}]$. Let either $s>1/p$ and $q\in(0,\fz]$, or $s=1/p$ and $q\in(0,1]$.
Then there exists a positive constant $C$ such that for all
$f\in R\bt$ and all $|x|\ge1$, it holds that
\begin{equation}
\label{eq-erg1}
|f(x)|\le C \, \|f\|_{\bt}\,  |x|^{\frac{1-n}p+n\tau}.
\end{equation}

{\rm (iii)} Let $\tau \in [0,\frac{n-1}{np}]$.
Let either $s>1/p$ and $q\in(0,\fz]$, or $s=1/p$ and $q\in(0,1]$.
In addition,  assume that $s \le n(\frac 1p - \tau)$.
Then there exists a positive constant $C$ such that for all
$f\in R\bt$ and all $x$, $0<|x|<1$, it holds that
\begin{equation}\label{eq-erg1d}
|f(x)|\le C \, \|f\|_{\bt}\,  \begin{cases}
1-\log_2|x|,\  & \tau=\frac{n-1}{np};\\
|x|^{\frac{1-n}p+n\tau}, \ \ & \tau\in[0,\frac{n-1}{np}).
 \end{cases}
\end{equation}
\end{theorem}

\begin{proof}
{\em Step 1.} Proof of (i). For the results from the continuous embedding of $\bt$ into a certain H\"older-Zygmund space
if  $s+n(\tau-1/p)>0$, see Proposition \ref{grundp}. Of course, the restriction to radial functions is superfluous here.

{\em Step 2.}
Part (ii) follows from the estimate (\ref{conv1}) in combination with the embeddings
for Besov-type spaces mentioned in Remark \ref{grund}.

{\em Step 3.}
Part (iii) follows from the estimate (\ref{conv3}), again in combination with the
embeddings for Besov-type spaces mentioned in Remark \ref{grund}.
\end{proof}

\begin{remark}
(i) Let us shortly comment on the behavior near infinity.
It is an easy exercise  in Fourier analysis to check that the function $f \equiv 1$
belongs to $\bt$ if and only if $0 < p \le \infty$, $\tau \ge 1/p$\, ,
$s \in \rr$ and $0 < q \le \infty \, .$
Since this function is radial it is immediate that functions belonging
to $R\bt$ need not  have decay near infinity.
With this respect the inequality (\ref{eq-erg1c}) is optimal in case $\tau \ge 1/p$.
Next we consider the particular case $\tau = \frac{n-1}{np}$. Then we have
\[
\frac{1-n}p+n\tau =0 \qquad \mbox{and}\qquad
s+n(\tau-1/p)>0 \qquad \Longleftrightarrow \qquad s> \frac 1p \, .
\]
Hence, the passage from Theorem \ref{t3.2}(i) (global boundedness and no decay)
to Theorem \ref{t3.2}(ii) (decay near infinity) is ``continuous''.

(ii) We comment on the behavior near the origin.
First of all, let us mention that in case $s> n(\frac 1p - \tau)$ we have global boundedness of all functions in
$\bt$. Hence, the assumption $s\le  n(\frac 1p - \tau)$ in
Theorem \ref{t3.2}(iii) is natural in our context.
In case $\tau =0$ and $s=n/p$ one knows even more about the
behavior of radial functions near the origin; see
\cite{ssv}.
For $\tau\in[0,\frac{n-1}{np})$ the bound near infinity is also the bound near the origin,
decay near infinity of order $|x|^{\frac{1-n}p+n\tau}$ becomes controlled unboundedness of order
$|x|^{\frac{1-n}p+n\tau}$ near the origin.
Again the limiting situation $\tau = \frac{n-1}{np}$ is of certain interest.
Then singularities of logarithmic order are possible.

(iii) For fixed $p$  and $\tau$ the largest space $\bt$ in (ii) and (iii)
of Theorem \ref{t3.2} is given by $B^{1/p, \tau}_{p,1} (\rn)$; see Remark \ref{grund}.

(iv)
There is a number of references with respect to the classical case $\tau =0$.
Following Strauss \cite{St} and Coleman, Glazer and Martin
\cite{CGM},
who had proved the inequality (\ref{eq-erg1}) for radial
functions belonging to
$W^1_2 (\rn)$, P.L.~Lions \cite{Li} gave the extension to all
first order Sobolev spaces $W^1_p(\rn)$.
The extension to Besov spaces and in particular, the existence of a
bound for $s$ for the  validity of these statements
have been found in Sickel and Skrzypczak \cite{ss00}.
More detailed investigations have been undertaken in
Sickel,  Skrzypczak and Vybiral \cite{ssv}.
Let us refer also to Cho and Ozawa \cite{CO}.
These authors  deal with $H^s(\rn)= B^s_{2,2}(\rn)$, $s>1/2$, using methods from Fourier analysis.
The advantage of this approach consists in its simplicity and the fact
that more information about the constant $C$ in  (\ref{eq-erg1}) is given.
In the framework of Morrey spaces those investigations seem to be new.

(v) Decay properties of functions under symmetry conditions have been
investigated at several places. We refer to Lions \cite{Li}, Kuzin and
Pohozaev \cite{KP}, Skrzypczak \cite{s02} and the references given there.
There also different types of symmetry constraints are investigated, e.g., block radial symmetry.
\end{remark}

Next we turn to the questions around sharpness in Theorem \ref{t3.2}.
Here we have some partial answers.

\begin{proposition}\label{p3.3}
Let $p\in(0,\fz)$ and $\tau \in [0,\fz)$.

{\rm(i)} Let $p>\frac{n-1}n$ and $\tau\in[0,\frac{n-1}{np}]$.
Assume that either $s<1/p$ and $q\in(0,\fz]$,
or $s=1/p$ and $q\in(1,\fz]$. Then for all $|x|\ge 1$, there exists a sequence
$\{f_N\}_{N\in\nn}$ of smooth and compactly supported radial functions, depending
on $x$, such that $\|f_N\|_{\bt}=1$ and
$\lim_{N\to\fz} |f_N(x)|=\fz$.

{\rm (ii)} Let $s<1/p$ and $q\in(0,\fz]$.
Then there exist a sequence $\{x_N\}_{N\in\nn}$ of points in $\rn$
and a corresponding sequence $\{f_N\}_{N\in\nn}$ of smooth and compactly
supported radial functions such that $x_N\to 0$, $\|f_N\|_{\bt}\le 1$ and
$|f_N(x_N)|\ge N|x_N|^{\frac{1-n}p-n\tau}$.
\end{proposition}

\begin{proof}
{\em Step 1.} Proof of (i).
We follow  ideas used in \cite[pp.\,651-653]{ss00}.
Let $\psi\in C_0^\fz(\rn)$ such that $0\le \psi\le 1$, $\psi (x) = 1$ if
$|x|\le 12$ and $\psi (x) =0$ if  $|x|\ge 13$. Denote by
$y_{j,k,\ell}$ the center of the ball $\Omega^R_{j,k,\ell}$ and let
$$\wz{\psi}_{j,0}(\cdot):=\psi(2^j\cdot)\quad{\rm and}\quad
\wz\psi_{j,k,\ell}(\cdot):=\psi(2^j(\cdot-y_{j,k,\ell}))$$
for all $j\in\nn_0$, $k\in\nn$ and $\ell\in\{1,\ldots,C(n,k)\}$.
Then by Lemma \ref{l3.1}, we see that
$$1\le \wz{\psi}_{j,0}+\sum_{k\in\nn}\sum_{\ell=1}^{C(n,k)}
\wz\psi_{j,k,\ell}\le Mm_0,$$
where $m_0$ denotes the multiplicity of the sequence of coverings,
$\{\Omega^R_{j,k,\ell}\}_{\ell}$, and $M$ is a positive constant independent of $j$.
Define
$$\psi_{j,0}:=\frac{\wz{\psi}_{j,0}}{\wz{\psi}_{j,0}
+\sum_{m\in\nn}\sum_{r=1}^{C(n,m)}
\wz\psi_{j,m,r}}$$
and
$$\psi_{j,k,\ell}:=\frac{\wz{\psi}_{j,k,\ell}}{\wz{\psi}_{j,0}
+\sum_{m\in\nn}\sum_{r=1}^{C(n,m)}
\wz\psi_{j,m,r}}\, .$$
Then
$${\psi}_{j,0}
+\sum_{m\in\nn}\sum_{r=1}^{C(n,m)}
\psi_{j,m,r}\equiv1.$$
Let $\psi_0\in C^\fz([0,\fz))$ be such that $0\le \psi_0\le 1$, $\psi_0 (x)=1$
if $0 \le x \le 1$ and $\psi_0 (x) = 0$ if $x \ge 2$. Define
\[
\vz_{j,r}(x):= \psi_0\lf(\lf|2^j|x|-2^r\r|\r)\, , \qquad x\in\rn \, ,
\]
for all $j,r\in\zz$.
Obviously, these functions are radial,
\begin{equation}\label{e5}
\supp \varphi_{j,r} \subset \{x\in\rn: \quad \max (0, \,
2^{-j}(2^r-2)) \le |x| \le 2^{-j}(2+2^r)\} \, ,
\end{equation}
and
\begin{equation}\label{e6}
\varphi_{j,r} (x) = 1 \qquad \mbox{if} \quad |x|= 2^{r-j} \, .
\end{equation}
For this family of functions it has been  proved in \cite[p.\,652]{ss00} that
\[
\vz_{j,r}=\sum_{k=2^r-r_0}^{2^r+r_0}\, \sum_{\ell=1}^{C(n,k)} \psi_{j,k,\ell}\, \vz_{j,r}
\qquad \mbox{if} \qquad r\ge \max(j+4,5)\,,
\]
where  $r_0$ is independent of $j$ and $r$.
Assume that $n\, (\max\{0, \frac1p-1\}) < s <  1/p \, .$
In this region no moment conditions are needed for our atoms.
For some positive constant $C_1$, the functions
$C_1\, 2^{-j(s-n/p)}\, \psi_{j,k,\ell}\, \vz_{j,r}$
are $(s,p)_{L,-1}$-atoms
with respect to $\Omega^R_{j,k,\ell}$; see \cite[p.\,652]{ss00}.
Then by Theorem \ref{t3.1} we know that
\begin{eqnarray}\label{e7}
\|\vz_{j,r}\|_{\bt}
&=&
\lf\|\sum_{k=2^r-r_0}^{2^r+r_0}\sum_{\ell=1}^{C(n,k)} \, \psi_{j,k,\ell}\, \vz_{j,r}\r\|_{\bt}
\\
& \asymp & \sup_{\gfz{P\in\mathcal{Q}}{j_P\le j}}
\frac1{|P|^\tau}\lf[\sum_{k=2^r-r_0}^{2^r+r_0} \omega (P,j,k) \, 2^{j(s-n/p)p}\r]^{1/p} \, ;
\nonumber
\end{eqnarray}
see (\ref{deco2}). From our knowledge about the  support of $\vz_{j,r}$ (see (\ref{e5})),
it becomes clear that the supremum is realized by dyadic
cubes with side-length $\le  C\, (2^{r} + r_0) \, 2^{-j}$
for some positive constant $C$ independent of $r$ and $j$.
Hence, the supremum in the previous formula runs only over those dyadic cubes
$P$ such that
$$2^{-j}\le  \ell(P)\ls (2^r+r_0)2^{-j}\ls 2^{r-j} \, .$$
If $j_P=j-m$ for some $m \in \nn_0$, then for those cubes $m$ satisfies  $0\ls  m\ls r$.
Finally, the number of  balls $\Omega^R_{j,k,\ell}$ satisfying
$\Omega^R_{j,k,\ell}\cap P\neq \emptyset$ is at most a constant multiple of $2^{m(n-1)}$, i.\,e.,
\begin{equation}\label{e8}
\omega (P,j,k) \ls 2^{m(n-1)}\, .
\end{equation}
To prove this claim we proceed as follows.
Let $12\, P_{j,k}$ denote the the annulus
$$\lf\{x\in\rn:\ \lf(k-\frac{11}2\r) 2^{-j}\le |x|<\lf(k+\frac{13}2\r) 2^{-j}\r\}\hs{\rm if}\hs k>
\frac{11}2,$$
or the ball $B(0, (k+\frac{13}2) 2^{-j})$ if $k\le
\frac{11}2 $. From $\diam\, \Omega^R_{j,k,\ell} = 12 \cdot 2^{-j}$ and the position of the centers $y_{j,k,\ell}$
of the balls $\Omega^R_{j,k,\ell}$ in Lemma
\ref{l3.1}(a), we deduce that
$$\bigcup_{\ell=1}^{C(n,k)} \Omega^R_{j,k,\ell} \subset 12P_{j,k}.$$
On the other hand, notice that the volume of $P\cap (12P_{j,k})$ is at most a
constant multiple of $2^{-j}\times 2^{-j_P(n-1)}=2^{-jn}2^{m(n-1)}$ (which can be attained
in the case $\diam (P\cap 12P_{j,k})\asymp \sqrt{n}\ell(P)$). Since
$|\Omega^R_{j,k,\ell}| \asymp 2^{-jn}$ and
$\{\Omega^R_{j,k,\ell}\}_{\ell\in\{1,\ldots,C(n,k)\}}$ is of finite multiplicity,
it follows  that there exists at most a constant multiple of $2^{m(n-1)}$ balls in
$\{\Omega^R_{j,k,\ell}\}_{\ell\in\{1,\ldots,C(n,k)\}}$ such that
$$\Omega^R_{j,k,\ell}\cap P=\Omega^R_{j,k,\ell}\cap (P\cap 12P_{j,k})\neq \emptyset.$$
Thus, \eqref{e8} holds true.
In the picture below we draw a cube $P$, $\ell (P) = 2^{-j_P}>2^{-j}$ in a position,
where $\omega (P,j,k)$ becomes maximal within the family of dyadic cubes $Q$ of side-length $2^{-j_P}$.

\begin{center}
\includegraphics[height=8cm]{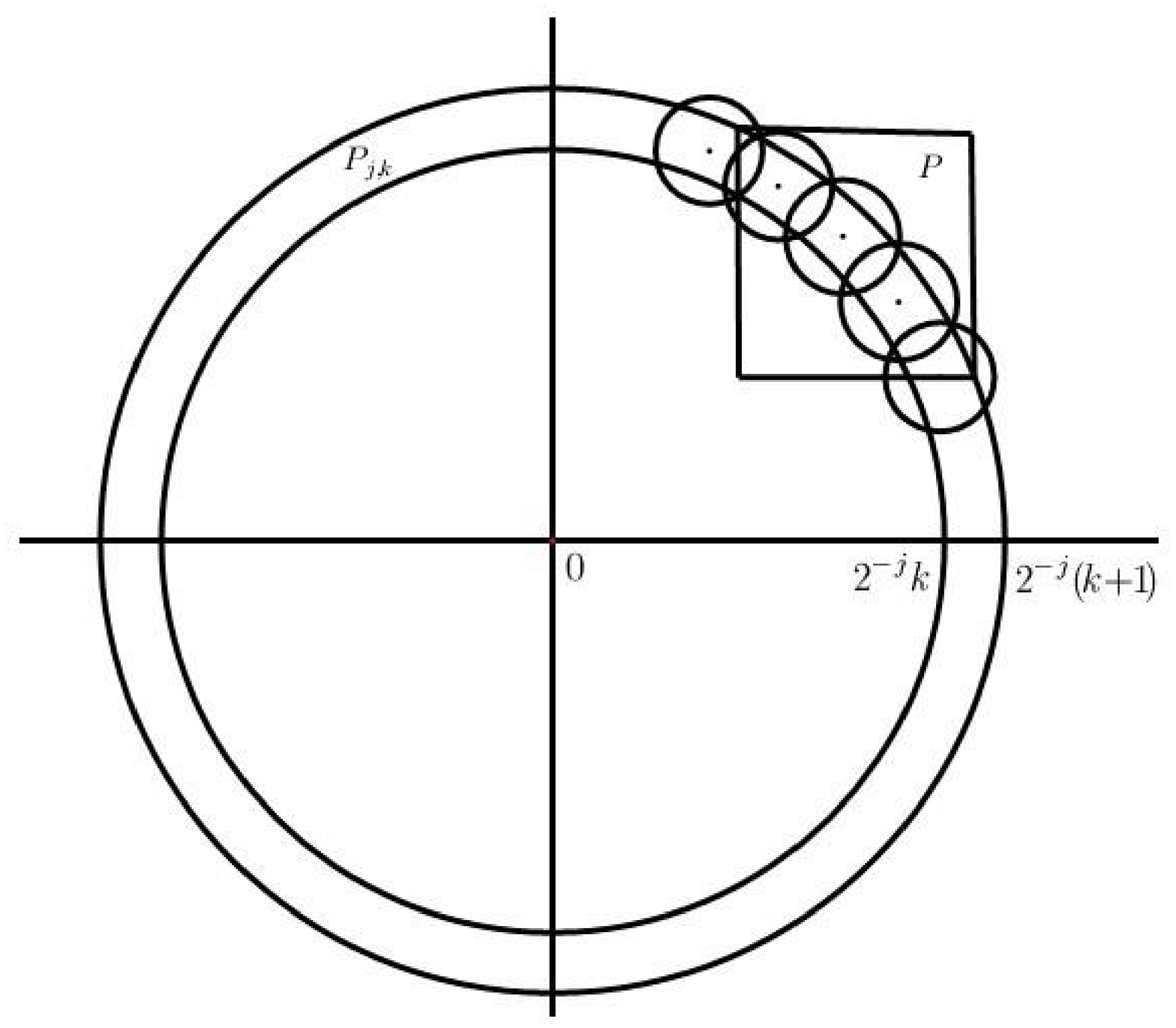}

\small{Fig.2\hs\hs A dyadic cube $P$ in a position with $\omega(P,j,k)$ large}
\end{center}

Therefore, employing $\tau \le  (n-1)/(np)$, (\ref{e7}) and (\ref{e8}), we obtain
\begin{eqnarray*}
\|\vz_{j,r}\|_{\bt}
\ls \sup_{0\le m\ls r} \, 2^{(j-m)n\tau} \, 2^{j(s-n/p)} \, 2^{m(n-1)/p}
\asymp  \,  2^{j(s-n/p+n\tau)} \, 2^{r(\frac{n-1}p-n\tau)}  \, .
\end{eqnarray*}
Now we turn to a new normalization
\[
\|\, 2^{-j(s-n/p+n\tau)} \, 2^{-r(\frac{n-1}p-n\tau)}\, \vz_{j,r}\|_{\bt}\le 1 \, ,
\]
but, thanks to (\ref{e6}), we find that
\begin{eqnarray*}
 2^{-j(s-n/p+n\tau)} \, 2^{-r(\frac{n-1}p-n\tau)}\, \vz_{j,r} ((2^{r-j},
 0,\, \, \ldots \, ,0)) & = & 2^{-j(s-n/p+n\tau)} \, 2^{-r(\frac{n-1}p-n\tau)}\\
&=& 2^{j(1/p-s)} \, 2^{(r-j)(-\frac{n-1}p+n\tau)} \, .
\end{eqnarray*}
Let  $|x|:=2^{t}$ for some $t\ge 5$. We apply the preceding
formulas with $t:=r-j$ fixed and $j$ tending to infinity.
This proves the desired result in the case of dyadic values of $|x|$.
The general case $|x|\ge 1$ follows from using the same arguments with some
simple modifications. Moreover, from the monotonicity of the quasi-norms
with respect to $s$ and $q$, we also obtain the result in the case  $s\le n(\max\{0,1/p-1\})$.

Now we consider the case that $s=1/p$ and $q\in(1,\fz]$. By the assumption on $p$
we know that $s>n(\max\{0,1/p-1\})$, which means that the atoms in
Theorem \ref{t3.1} do not need to satisfy moment conditions.
Let
$$\vz_{N,r}=\sum_{j=1}^{N}\, \az_j\, \vz_{j,j+4+r}\, , \qquad N\in\nn\, , \quad r\ge 0\, ,$$
 where $\{\az_j\}_j$ is a sequence of positive numbers such that $\{\az_j\}_j\in \ell^q$  but $\{\az_j\}_j\not\in \ell^1$.
Then, similar to the above argument in case $s<1/p$, we see that
\[
\|\vz_{N,r}\|_{\bt}
\asymp  \sup_{P\in\mathcal{Q}} \,
\frac1{|P|^\tau} \left\{ \sum_{j= \max(j_P,1)}^N \,  \az_j^q \,
\lf[ \sum_{k= 2^{j+4+r}-r_0}^{2^{j+4+r}+r_0} \, \omega (P,j,k) \, 2^{j(1-n)}\r]^{q/p}
\r\}^{1/q}\, .
\]
Arguing as above, in the supremum  only those dyadic cubes
$P$, such that $2^{-N}\le  \ell(P)\ls 2^{r} $, are of relevance; see (\ref{e5}). Thus,
$-r\ls j_P\le N$. By the same argument as used in \eqref{e8}, we see that,
for all $j\ge j_P$, $\omega (P,j,k) \ls \, 2^{(j-j_P)(n-1)}$.
Therefore, it follows from $\tau\in[0,\frac{n-1}{np}]$ that
\begin{eqnarray*}
\|\vz_{N,r}\|_{\bt}
&&\ls \sup_{\gfz{P\in\mathcal{Q}}{-r\ls j_P\le N}} \,
2^{j_P[n\tau-(n-1)/p]} \left\{ \sum_{j= \max(j_P,1)}^N \,  \az_j^q \,
\r\}^{1/q}
\ls  2^{r(\frac{n-1}p-n\tau)}\lf\{\sum_{j=1}^N \az_j^q\r\}^{1/q}.
\end{eqnarray*}

Next, observe that
\[
\vz_{N,r}((2^{r+4},0, \, \ldots \, ,0)) = \sum_{j=1}^N\az_j\to\fz  \qquad \mbox{as} \quad N\to\fz \, ;
\]
see (\ref{e6}). By taking $r$ fixed, but letting $N$ tend to infinity, this finishes
the proof for the case of dyadic values of $|x|$.
The general case $|x|\ge 1$ follows from some simple modifications.

{\em Step 2.} Proof of (ii).
Assume first that $n(\max\{0,\, 1/p-1\})<s<1/p$. Since $2^{-j(s-n/p)}\vz_{j,0}$
is a constant multiple of an $(s,p)_{L,-1}$-atom with respect to
$\Omega^R_{j,0}$, Theorem \ref{t3.1} yields
$$\|2^{-jn\tau}2^{-j(s-n/p)}\vz_{j,0}\|_{B^{s,\tau}_{p,q}(\rn)}\ls 1.$$
Considering the function $2^{-jn\tau}2^{-j(s-n/p)}\vz_{j,0}$ on the ring $|x|=2^{-j}$,
we know that
$$2^{-jn\tau}2^{-j(s-n/p)}|\vz_{j,0}(x)|=|x|^{s-n/p+n\tau}\ge N |x|^{\frac{1-n}p+n\tau}$$
when $|x|$ is small enough. The general case $s<1/p$ follows from a monotonicity
argument. This finishes the proof of Proposition \ref{p3.3}.
\end{proof}

\begin{remark}
(i) Combining Proposition \ref{p3.3} and Theorem \ref{t3.2}(ii), we know that
the inequality (\ref{eq-erg1}) can not be true if either $s<1/p$
and $q$ arbitrary, or $s=1/p$ and $1 < q \le \infty$.

(ii) Proposition \ref{p3.3}(ii) in combination with Theorem \ref{t3.2}(iii)
yields  that in case $s<1/p$ the inequality (\ref{eq-erg1d}) can not be true.

(iii) In case $\tau =0$ a much more detailed investigation of the sharpness of the results
has been undertaken in \cite{ssv}.
However, the approach used there is based on a characterization of the trace spaces of
$RB^s_{p,q}(\rn)$. That would be of certain interest also in the present situation.
\end{remark}

\subsection*{Smoothness outside the origin, decay near infinity and controlled unboundedness of radial functions
belonging to Sobolev-Morrey spaces}

\hskip\parindent
For the convenience of the reader we formulate the consequences of
Theorem \ref{tc.1}, Theorem \ref{t3.2}
and Proposition \ref{p3.3} in case of Sobolev-Morrey spaces.
Applying the embedding
\begin{equation}\label{ebd}
W^m \cm^u_p (\rn) \hookrightarrow   \cn^{m}_{u, p,\infty} (\rn)
= B^{m,\tau}_{p,\infty} (\rn)\, , \qquad \tau = \frac 1p - \frac 1u\, ,
\end{equation}
and Theorem \ref{tc.1}, we have the following result.

\begin{corollary}\label{cc.1}
Let $n \ge 2$, $1 < p \le u < \fz$ and $m \in \nn$.
Every radial element $f $ in $W^m \cm_p^u (\rn)$
has a representative $\widetilde{f}$ which is continuous outside of the origin.
\end{corollary}

The next result follows from \eqref{ebd} and Theorem \ref{t3.2}.

\begin{corollary}\label{c3.2}
Let $n\ge 2$, $1 <  p \le u < \fz$ and $m \in \nn$.

{\rm (i)} Let $u/p \le n$.
Then there exists a positive constant $C$ such that for all
$f\in RW^m\cm_p^u (\rn)$ and all $|x|\ge 1$, it holds that
\begin{equation}
\label{eq-erg1f}
|f(x)|\le C \, \|f\|_{W^m\cm_p^u (\rn)}\,  |x|^{\frac{1}p - \frac nu}.
\end{equation}

{\rm (ii)} Let $u/p \le n$.
Then there exists a positive constant $C$ such that for all
$f\in RW^m \cm_p^u (\rn)$ and all $x$ with $0<|x|<1$, it holds that
\begin{equation}\label{eq-erg1g}
|f(x)|\le C \, \|f\|_{W^m\cm_p^u (\rn)}\,  \begin{cases}
1-\log_2|x|,\  & \frac 1p=\frac nu;\\
|x|^{\frac{1}p-\frac nu}, \ \ &  \frac 1p<\frac nu.
 \end{cases}
\end{equation}
\end{corollary}

\begin{remark}
Let us mention that because of the embedding (\ref{morrey2})
only those situations are of interest where $n\ge u$ holds.
Of course, part (i) of Corollary \ref{c3.2} represents Theorem
\ref{t3.2b} proved in a totally different way.
\end{remark}

\section{Appendix - Proof of Theorem \ref{t2.1}}

\hskip\parindent
{\em Step 1.}
Recall that $B^{s,\tau}_{p,q}(\rn)=B^{s+n(\tau-1/p)}_{\fz,\fz}(\rn)$
if either $q\in(0,\fz)$ and $\tau\in(1/p,\fz)$, or
$q=\fz$ and $\tau\in[1/p,\fz)$ (see \cite{yy02}).
The desired conclusion in the above case and in case $\tau =0$ has been given in
\cite[Proposition 1]{ss00}.

{\em Step 2.}
Next we assume that either $\tau\in(0,1/p)$ and $q \in (0,\fz]$, or $\tau=1/p$ and $q\in(0,\fz)$.

{\em Substep 2.1.} Let $f \in \bt$ be given. We show
 the existence of an appropriate atomic decomposition.
Let $\Psi$, $\psi\in\cs(\rn)$ be such that
\[|\widehat{\Psi}(\xi)|>0\qquad  \mbox{if} \quad |\xi|\le 2 \qquad
\mbox{and} \qquad |\widehat{\psi}(\xi)|>0\quad \mbox{if} \quad 1/2\le|\xi|\le 2\, .
\]
Then there exist $\Phi,\,\vz\in\cs(\rn)$ and $\delta\in(0,\fz)$ such that
\[
\supp\widehat{\Phi} \subset \{\xi\in\rn:\ |\xi|\le2\} \qquad \mbox{and}
\qquad  |\widehat{\Phi}(\xi)|>0\quad \mbox{if} \quad  |\xi|\le \delta\, ,
\]
 $\vz$ satisfies \eqref{e1.1}, and
$$\widehat{\Psi}(\xi)\overline{\widehat{\Phi}(\xi)}+\sum_{j\in\nn} \widehat{\psi}(2^{-j}\xi)
\overline{\widehat{\vz}(2^{-j}\xi)}=1$$
for all $\xi\in\rn$; see Frazier and Jawerth \cite{fj85,fj90}.
Let $A_n$ be the same as in Remark \ref{Konstanten}. In addition we may assume that
\[
\supp \Psi \, , \quad \supp \psi \subset \{x \in \rn: \quad |x|\le A_n/2\}
\]
and
\[
\int_\rn x^\gamma \psi (x) \, dx =0 \qquad \mbox{for all} \quad |\gamma|\le M\, ;
\]
see \cite[Proposition~1]{ss00}.
Then, by the Calder\'on reproducing formula, we see that
\begin{eqnarray*}
f(x) = \int_{\rn}\Psi(x-y)\,  \wz\Phi\ast f(y)\,dy+\sum_{j\in\nn}
\int_{\rn} \psi_j(x-y) \, \wz\vz_j\ast f(y)\,dy
\end{eqnarray*}
in $\cs'(\rn)$, where $\wz\Phi(\cdot)=\oz{\Phi(-\cdot)}$
and $\wz\vz_j(\cdot)=\oz{\vz_j(-\cdot)}$.
For all $j,\,\ell\in\nn_0$, we define
\begin{equation}\label{e2.3}
t_{j,\ell}:=
\begin{cases} D(n,M)\dsup_{y\in\Omega_{0,\ell}}|\wz\Phi\ast f(y)|, & j=0;
\\
E(n,M) \, 2^{j(s-n/p)}\dsup_{y\in\Omega_{j,\ell}}|\wz\vz_j\ast f(y)|,\quad &j\in\nn,
\end{cases}
\end{equation}
where
\begin{eqnarray*}
D(n,M):= B^n_nw_n\left[\max_{|\az|\le L} \|D^\az \Psi\|_{C(\rn)}\r]
\end{eqnarray*}
and
\begin{eqnarray*}
E(n,M):=&&B^n_nw_n\left[\max_{|\az|\le L} \|D^\az \psi\|_{C(\rn)}\r] \\
&&\times\max_{|\az|\le L}\left\{A_n^{-s+n/p+|\az|}, B_n^{-s+n/p+|\az|}\r\}
\max\lf\{1,w_n(3/2)^{M+n+1}\sum_{|\gz|=M+1}\frac1{\gz!}\r\}.
\end{eqnarray*}
As in \cite[Proposition~1]{ss00}, for all $j\in\nn_0$, we let
$\Omega^*_{j,0}:=\Omega_{j,0}$ and
$$\Omega^*_{j,\ell}:=\Omega_{j,\ell}\setminus \left(
\bigcup_{m=0}^{\ell-1}\Omega_{j,m}\right)\quad \mathrm{if}\quad \ell\in\nn.$$
Then $\{\Omega^*_{j,\ell}\}_{\ell\in\nn_0}$ is a family of pairwise disjoint sets
and satisfies that $\cup_{\ell\in\nn_0} \Omega^*_{j,\ell}=\rn$ for all
$j\in\nn_0.$
When $t_{j,\ell}\neq 0$ we let
\begin{equation}\label{e2.4}
a_{j,\ell}(x):=
\begin{cases} \dfrac1{t_{0,\ell}}\int_{\Omega^*_{0,\ell}}\Psi(x-y)\, \wz\Phi\ast f(y)\,dy, & j=0;\\
&\\
\dfrac{1}{t_{j,\ell}}\int_{\Omega^*_{j,\ell}} \psi_j(x-y)\, \wz\vz_j\ast f(y)\,dy,\quad &j\in\nn.
\end{cases}
\end{equation}
When $t_{j,\ell}=0$ we define $a_{j,\ell}(x)\equiv0$.
Then, by the argument used in the proof of \cite[Proposition~1]{ss00},
we know that  $a_{0,\ell}$ is an $1_L$-atom centered at the set $\Omega_{0,\ell}$,
$a_{j,\ell}$ with $j\in\nn$ an $(s,p)_{L,M}$-atom centered at the set $\Omega_{j,\ell}$
and $f=\sum_{j=0}^\fz\sum_{\ell=0}^\fz t_{j,\ell}\,a_{j,\ell}$
in $\cs'(\rn)$.
Now we  prove that
\begin{equation}\label{s1}
\|\{t_{j,\ell}\}_{j,\ell\in\nn_0}\|_{b(p,q,\tau,\Omega)}\ls \|f\|_{\bt} \, .
\end{equation}

By definition of the $t_{j,\ell}$, we find that
\begin{eqnarray*}
&&\|\{t_{j,\ell}\}_{j,\ell\in\nn_0}\|_{b(p,q,\tau,\Omega)}\\
&&\hs\asymp\,  \sup_{P\in\mathcal{Q}}\frac1{|P|^\tau}
\left\{\sum_{j=\max\{j_P,0\}}^\fz 2^{j(s-n/p)q}
\left[\sum_{\gfz{\ell\in\nn_0}{\Omega_{j,\ell}\cap P\neq \emptyset}}
\sup_{y\in \Omega_{j,\ell}}|\wz\vz_j\ast f(y)|^p\right]^{q/p}\right\}^{1/q}.
\end{eqnarray*}
Similar to the argument used in the proof of \cite[(2.11)]{fj85},
we conclude that
\begin{eqnarray*}
\sup_{y\in\Omega_{j,\ell}}|\wz\vz_j\ast f(y)|^p
\ls 2^{jn}\sum_{k\in\zz^n}(1+|k|)^{-N}\int_{\Omega_{j,\ell}}
|\wz\vz_j\ast f (y+2^{-j}k)|^p\,dy \, ,
\end{eqnarray*}
where $N$ is at our disposal.
On the other hand, since $\diam\Omega_{j,\ell}\le B_n2^{-j}$, there exists a constant $d\in[1,\fz)$ such that, for all
$j\ge j_P$,
$$\bigcup_{\gfz{\ell\in\nn_0}{\Omega_{j,\ell}\cap P\neq \emptyset}} \Omega_{j,\ell}
\subset dP.$$
Denote by $y_P$ the center of $P$ and by $B (y,R)$
the ball with radius $R$ and center in $y$. It follows that
\begin{eqnarray*}
&&\|\{t_{j,\ell}\}_{j,\ell\in\nn_0}\|_{b(p,q,\tau,\Omega)}\\
&&\hs\ls\sup_{P\in\mathcal{Q}}\frac1{|P|^\tau}
\left\{\sum_{j=\max\{j_P,0\}}^\fz 2^{jsq}
\left[\sum_{\gfz{\ell\in\nn_0}{\Omega_{j,\ell}\cap P\neq \emptyset}}\sum_{k\in\zz^n}\frac{\int_{\Omega_{j,\ell}}
|\wz\vz_j\ast  f (y+2^{-j}k)|^p\,dy}{(1+|k|)^{N}}
\right]^{q/p}\right\}^{1/q}\\
&&\hs\ls\sup_{P\in\mathcal{Q}}\frac1{|P|^\tau}
\left\{\sum_{j=\max\{j_P,0\}}^\fz 2^{jsq}
\left[\sum_{k\in\zz^n} \frac{\int_{dP}|\wz\vz_j\ast f(y+2^{-j}k)|^p\,dy}{(1+|k|)^{N}}
\right]^{q/p}\right\}^{1/q}\\
&&\hs\ls\sup_{P\in\mathcal{Q}}\frac1{|P|^\tau}
\left\{\sum_{j=\max\{j_P,0\}}^\fz 2^{jsq}
\left[\sum_{k\in\zz^n}\frac{\int_{B(y_P, d\ell(P)+|k|\ell(P))}|\wz\vz_j\ast f(z)|^p\,dz}{(1+|k|)^{N}}
\right]^{q/p}\right\}^{1/q}.
\end{eqnarray*}
Temporarily we assume  $q\le p$. By the definition of $\|\, f\, \|_{\bt}$, we conclude
that
\begin{eqnarray*}
&&\|\{t_{j,\ell}\}_{j,\ell\in\nn_0}\|_{b(p,q,\tau,\Omega)}\\
&&\hs\ls\sup_{P\in\mathcal{Q}}\frac1{|P|^\tau}
\left\{\sum_{k\in\zz^n}(1+|k|)^{-Nq/p}\r.\\
&&\lf.\hs\hs\times\sum_{j=\max\{j_P,0\}}^\fz 2^{jsq}
\left[\int_{B(y_P, d\ell(P)+|k|\ell(P))}|\wz\vz_j\ast f(z)|^p\,dz
\right]^{q/p}\right\}^{1/q}\\
&&\hs \ls\|f\|_{\bt}
\left\{\sum_{k\in\zz^n}(1+|k|)^{-Nq/p} (d+|k|)^{n\tau q}\right\}^{1/q}\ls \|f\|_{\bt},
\end{eqnarray*}
where we have chosen $N>n(\tau+1/q)p$. Now we turn to the case $q >p$.  With $N>n(\tau p+1)$ we obtain
\begin{eqnarray*}
&&\|\{t_{j,\ell}\}_{j,\ell\in\nn_0}\|_{b(p,q,\tau,\Omega)}\\
&&\hs\ls
\left\{\sum_{k\in\zz^n}(1+|k|)^{-N}\r.\\
&&\lf.\hs\hs\times
\sup_{P\in\mathcal{Q}}\frac1{|P|^{\tau p}}\left[\sum_{j=\max\{j_P,0\}}^\fz 2^{jsq}
\left[\int_{B(y_P, d\ell(P)+|k|\ell(P))}|\wz\vz_j\ast f(z)|^p\,dz
\right]^{q/p}\r]^{p/q}\right\}^{1/p}\\
&&\hs \ls\|f\|_{\bt}
\left\{\sum_{k\in\zz^n}(1+|k|)^{-N} (d+|k|)^{n\tau p}\right\}^{1/q}\ls \|f\|_{\bt} \, .
\end{eqnarray*}
Hence, we have proved that any $f \in \bt$ admits an atomic decompositions with respect to the sequence of coverings
$(\{\Omega_{j,\ell}\}_\ell)_j$ such that (\ref{s1}) holds with some constants behind $\ls$ independent of $f$.

{\em Substep 2.2.} Now we study the regularity of appropriate atomic decompositions with
coefficient sequence belonging to $b(p,q,\tau,\Omega)$.
Associated to the sequence of coverings $(\{\Omega_{j,\ell}\}_\ell)_j$
there is a regular sequence of coverings by balls $(\{B_{j,\ell}\}_\ell)_j$
such that
\[
\Omega_{j,\ell} \subset B_{j,\ell} \, , \qquad j \in \nn_0, \quad \ell \in \nn_0\, ;
\]
see \cite[Proposition~1]{ss00}.
Any atomic decomposition with respect to
$(\{\Omega_{j,\ell}\}_\ell)_j$ represents trivially an
atomic decomposition with respect to
$(\{B_{j,\ell}\}_\ell)_j$.
But for such a standard covering the desired property is known to be true, we refer to \cite[Theorem~3.3]{ysy}.
Let us mention that this has been  the only place, where we use $\tau \le 1/p$ in Step 2.
This finishes the proof of Theorem \ref{t2.1}.

\bigskip

Wen Yuan

\medskip

School of Mathematical Sciences, Beijing Normal University,
Laboratory of Mathematics and Complex Systems, Ministry of
Education, Beijing 100875, People's Republic of China

\medskip

and

\medskip

Mathematisches Institut, Friedrich-Schiller-Universit\"at Jena,
Ernst-Abbe-Platz 2, D-07743 Jena, Germany

\smallskip

{\it E-mail}: \texttt{wenyuan@bnu.edu.cn}

\bigskip

Winfried Sickel

\medskip

Mathematisches Institut, Friedrich-Schiller-Universit\"at Jena,
Ernst-Abbe-Platz 2, D-07743 Jena, Germany

\smallskip

{\it E-mail}: \texttt{winfried.sickel@uni-jena.de}

\bigskip

Dachun Yang

\medskip

School of Mathematical Sciences, Beijing Normal University,
Laboratory of Mathematics and Complex Systems, Ministry of
Education, Beijing 100875, People's Republic of China

\smallskip

{\it E-mail}: \texttt{dcyang@bnu.edu.cn}

\end{document}